\documentclass[a4paper, 11pt, reqno]{amsart}%

\usepackage{amssymb} 
\usepackage{amsmath}
\usepackage{amsthm}
\usepackage{amsfonts}    

\usepackage{graphicx}   

\theoremstyle{plain}
\newtheorem{theorem}{Theorem}[section]

\newtheorem{cor}[theorem]{Corollary}

\newtheorem{lem}[theorem]{Lemma}

\newtheorem{prop}[theorem]{Proposition}

\theoremstyle{definition}

\newtheorem{rem}[theorem]{Remark}
\numberwithin{equation}{section}
\numberwithin{theorem}{section}

\allowdisplaybreaks[4] 

\begin{document}

\title[Uniqueness for indefinite superlinear boundary condition]{Uniqueness of a positive solution for the Laplace equation with indefinite superlinear boundary condition} 

\author{Kenichiro Umezu} 

\address{Department of Mathematics, Faculty of Education, Ibaraki University, Mito 310-8512, Japan}

\email{\tt kenichiro.umezu.math@vc.ibaraki.ac.jp} 



\thanks{This work was supported by JSPS KAKENHI Grant Number JP18K03353.}

    \maketitle

\begin{abstract}
In this paper, we consider the Laplace equation with a class of indefinite superlinear boundary conditions. 
Superlinear elliptic problems can be expected to have multiple positive solutions by some case. Conducting spectral analysis for the linearized eigenvalue problem at an unstable positive solution, we find sufficient conditions for ensuring that the implicit function theorem is applicable to the one, and then deduce the uniqueness result for a positive solution. An application of our results to the logistic boundary condition arising from population genetics is given. 
\end{abstract}


\section{Introduction and main results} 

\label{intro}

Let $\Omega$ be a bounded domain of $\mathbb{R}^N$, $N\geq1$, 
with smooth boundary $\partial\Omega$. Consider 
positive solutions of the Laplace equation with indefinite superlinear boundary condition 
\begin{align}  \label{sppr}
\begin{cases}
-\Delta v = 0 & \mbox{ in } \Omega, \\
\dfrac{\partial v}{\partial \nu}  
= \lambda g(x)(v+v^p) & \mbox{ on } \partial\Omega,   
\end{cases} 
\end{align}
where $p>1$ is a given exponent, $\lambda>0$ is a parameter, $g\in C^{1+\theta}(\partial\Omega)$ changes sign, 
and $\nu$ is the unit outer normal to $\partial\Omega$. Throughout this paper, unless stated otherwise, 
we assume that $p>1$ is {\it subcritical}, i.e., 
\begin{align} \label{psubcr}
1<p<\frac{N}{N-2} \ \mbox{ if } N>2. 
\end{align} 
A nonnegative solution $u\in C^{2+\alpha}(\overline{\Omega})$, $\alpha\in (0,1)$, of \eqref{sppr} is called \textit{positive} if $u\not\equiv 0$. 
Using the strong maximum principle and Hopf's boundary point lemma (\cite{PW67}), a positive solution of \eqref{sppr} is positive in $\overline{\Omega}$. 
For the existence of a positive solution, we can show that if $\int_{\partial\Omega}g\geq0$, 
no positive solution of \eqref{sppr} exists for any $\lambda>0$ (Proposition \ref{prop:non}). Therefore, in terms of the existence, we focus our consideration of \eqref{sppr} on the case when $\int_{\partial\Omega}g<0$.

When $\int_{\partial\Omega}g<0$, we call 
$\lambda_1(g)>0$ the {\it positive principal eigenvalue} (smallest positive eigenvalue) of the linear eigenvalue problem 
\begin{align} \label{epro01}
\begin{cases}
-\Delta \varphi = 0 & \mbox{ in } \Omega, \\
\dfrac{\partial \varphi}{\partial \nu}  
= \lambda g(x) \varphi & \mbox{ on } \partial\Omega. 
\end{cases}
\end{align}
Here, an eigenvalue of \eqref{epro01} is called {\it principal} if the eigenfunctions associated with it have constant sign. 
In fact, \eqref{epro01} has exactly two principal 
eigenvalues $\lambda=0, \lambda_1(g)$ if $\int_{\partial\Omega}g<0$, 
which are both simple and possess eigenfunctions that are positive in $\overline{\Omega}$ (\cite{Um2006}). We call $\varphi_1(g)$ a positive eigenfunction associated with $\lambda_1(g)$ 
(nonzero constants are the eigenfunctions associated with the principal eigenvalue $\lambda=0$). We note that 
the smallest eigenvalue of $(-\Delta, \frac{\partial}{\partial \nu}-\lambda g(x))$ is positive for $\lambda \in (0,\lambda_1(g))$ (Lemma \ref{lem:coer}). 
For the case of $\int_{\partial\Omega}g\geq0$, 
we know that \eqref{epro01} has no positive principal eigenvalue (i.e., zero is a unique nonnegative principal eigenvalue), thus, it is understood that $\lambda_1(g)=0$, and $\varphi_1(g)$ is a positive constant. 

In this paper, 
we aim to give a class of $g$ satisfying $\int_{\partial\Omega}g<0$ in which \eqref{sppr} has a {\it unique}  positive solution $v_\lambda$ for each  $\lambda \in (0,\lambda_1(g))$ 
and no positive solution for any $\lambda\geq\lambda_1(g)$. 
Moreover, we show that $v_\lambda\in C^{2}(\overline{\Omega})$ is parametrized smoothly by $\lambda\in (0,\lambda_1(g))$, bifurcates from the {\it constant line} $\{ (\lambda, 0 )\}$ at $\lambda=\lambda_1(g)$, satisfies $\| v_\lambda \|_{C(\overline{\Omega})} \rightarrow \infty$ as $\lambda\searrow0$, and is unstable (Figure \ref{figcurve}). In Section \ref{sec:Appl}, we apply this result to the Laplace equation with the logistic boundary condition, \eqref{lpr}.  
%
		 \begin{figure}[!htb]
    \centering 
	\includegraphics[scale=0.16]{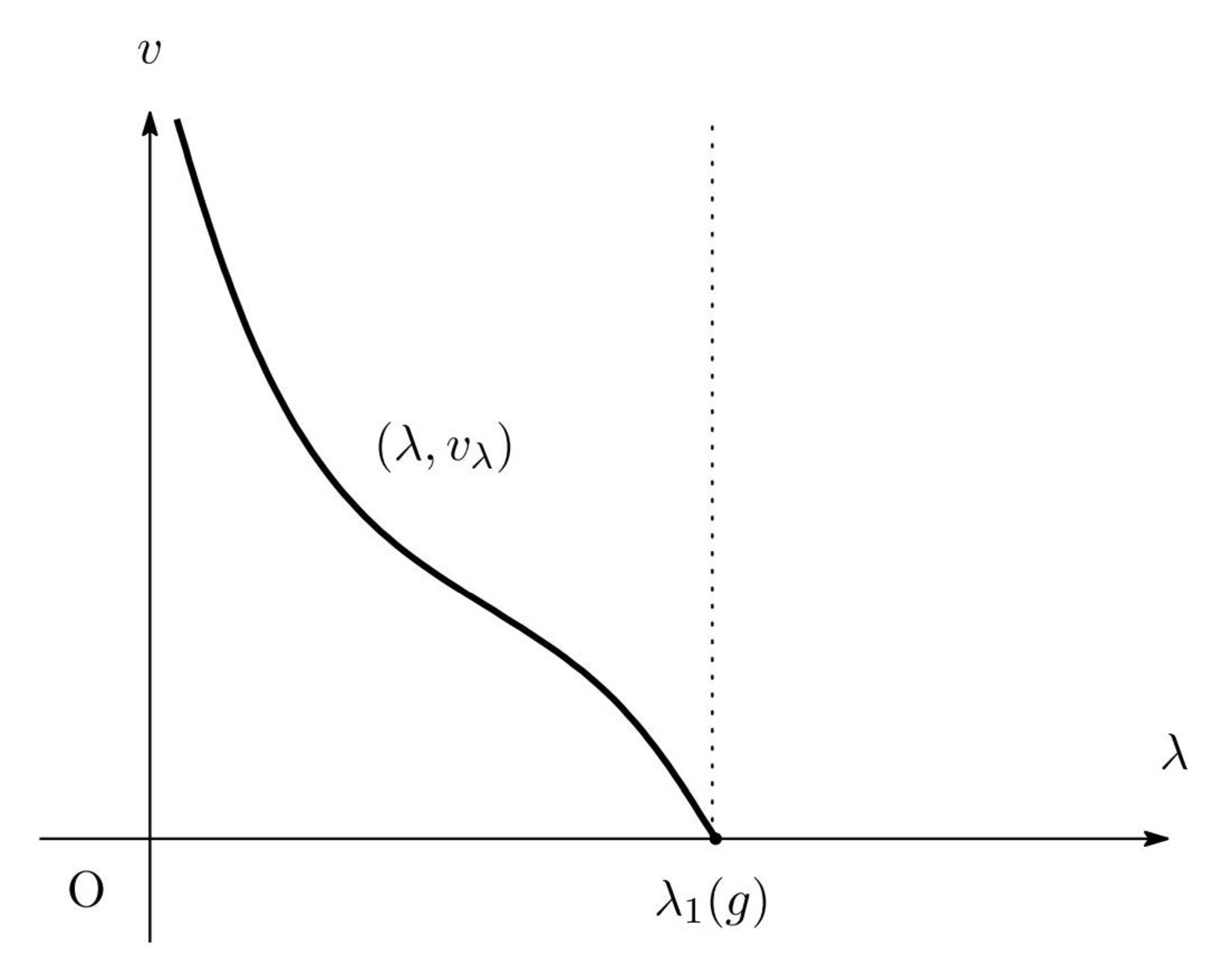} 
		  \caption{Positive solution set $\{ (\lambda,v_{\lambda})\}$ of \eqref{sppr} as a smooth curve.} 
		\label{figcurve} 
		    \end{figure} 

For our aim, the change of variables 
\begin{align} \label{change}
w=\lambda^{\frac{1}{p-1}}v
\end{align}
is useful, because it transforms \eqref{sppr} into the problem 
\begin{align} \label{pr:w}
\begin{cases}
-\Delta w = 0 & \mbox{ in } \Omega, \\  
\dfrac{\partial w}{\partial \nu}=\lambda g(x)w + g(x)w^p & 
\mbox{ on } \partial\Omega. 
\end{cases}
\end{align}
Before discussing the uniqueness issue, we consider the general case of $g$, and use the fibering map method to establish the existence of positive solutions and their properties for \eqref{pr:w} as follows:
\begin{theorem} \label{thm1}
Assume that $\int_{\partial\Omega}g(x)<0$. Then, \eqref{pr:w} possesses 
a positive solution $w_\lambda$ 
for every $\lambda\in (0,\lambda_1(g))$ and no positive solution for any $\lambda\geq \lambda_1(g)$. Furthermore, $w_{\lambda}$ satisfies the following: 
\begin{enumerate} \setlength{\itemsep}{0.2cm} 
\item $w_\lambda$ is unstable, 
\item $\| w_\lambda \|_{C^{2}(\overline{\Omega})}\rightarrow 0$ 
as $\lambda\nearrow \lambda_1(g)$, 
\item $\sup_{\lambda\in (0, \lambda_1(g))}\| w_\lambda \|_{C(\overline{\Omega})} <\infty$, and 
\item $\inf_{\lambda\in (0, \overline{\lambda})}\| w_\lambda \|_{C(\overline{\Omega})}>0$ for any $\overline{\lambda}
\in (0,\lambda_1(g))$. 
\end{enumerate}
\end{theorem}

We refer to \cite{Zh2003} for the existence of positive solutions to a similar type of nondivergence elliptic problem with indefinite superlinear boundary conditions. 
For nonlinear elliptic equations with definite superlinear boundary conditions, we refer to the survey article \cite{Ro2005}. We refer to \cite[Theorem 3]{Pf1999} for a similar existence result for positive solutions of the Neumann problem 
\begin{align} \label{Npr}
\begin{cases}
-\Delta w = \lambda m(x) w + m(x) w^{p} & \mbox{ in } \Omega, \\
\dfrac{\partial w}{\partial \nu}=0 & \mbox{ on } \partial\Omega.  
\end{cases}
\end{align}
Here, $p>1$ is subcritical, i.e., $p<\frac{N+2}{N-2}$ if $N>2$, 
$\lambda>0$ is a parameter, 
$m\in C^{\theta}(\overline{\Omega})$ changes sign such that 
$\int_\Omega m < 0$, and $(-\Delta-\lambda m(x), \frac{\partial}{\partial \nu})$ has a positive smallest eigenvalue. 
See also \cite{AT93, ALG98, BCDN94, BCDN95, CL97, Ou91, Te96} 
for the existence and related issues for positive solutions of the analogous indefinite superlinear elliptic equations with linear Dirichlet or Neumann boundary conditions.  

Next, we consider a special case of $g$ that is central to this paper. For a sign changing $g\in C^{1+\theta}(\partial\Omega)$, we set  
\begin{align} 
&\Gamma_{\pm}:= \{ x \in \partial\Omega : g(x)\gtrless0 \}, \label{defGam+} \\
&\Gamma_{0}:= \{ x\in \partial\Omega : g(x)=0 \}, \nonumber \\
&\Gamma_{-,0}:= \Gamma_{-}\cup \Gamma_{0}.  \nonumber 
\end{align}
It should be noted that $\Gamma_{+}$ is open in the relative topology of $\partial\Omega$. We then introduce the following condition for $\Gamma_{+}$ (Figure \ref{figPOmegat}):  
\begin{align} \label{G}
\mbox{$\Gamma_{+}$ is a compact submanifold 
of $\partial\Omega$ with dimension $N-1$.} 
\end{align}
For $g\in C^{1+\theta}(\partial\Omega)$ equipped with 
\eqref{G}, we define $g_{\delta}\in C^{1+\theta}(\partial\Omega)$ by the formula 
\begin{align} \label{gdel}
g_{\delta}:=g^{+}-\delta g^{-}, \quad 
\delta >\delta_{0}:=\frac{\int_{\partial\Omega}g^{+}}{\int_{\partial\Omega}g^{-}}.
\end{align}
We observe that $\int_{\partial\Omega}g_{\delta}<0$, and
$\int_{\partial\Omega}g_{\delta}\nearrow \int_{\partial\Omega}g_{\delta_{0}}=0$ as 
$\delta \searrow \delta_{0}$, and know that $\lambda_1(g_{\delta})\rightarrow 0$ as 
$\delta\searrow \delta_{0}$ (see \eqref{lam1del0}). 
%
		 \begin{figure}[!htb] 
        \centering 	  	   
	\includegraphics[scale=0.16]{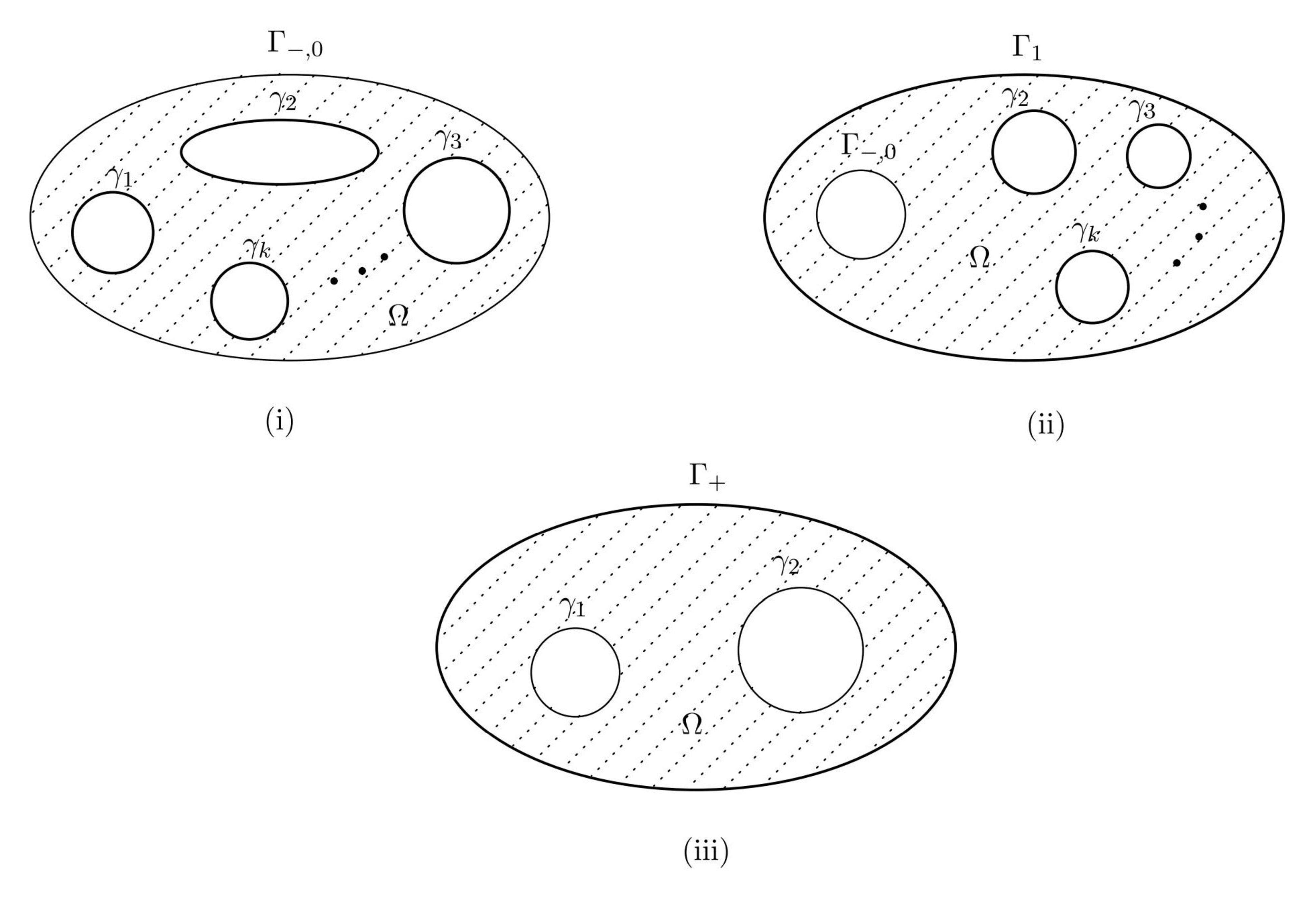} 
		  \caption{Situations admitting \eqref{G}: 
		  (i) $\Gamma_{+}=\bigcup_{j=1}^{k}\gamma_k$, 
		  (ii) $\Gamma_{+}=\Gamma_1 \cup \left( \bigcup_{j=2}^{k}\gamma_k \right)$, 
		  and 
		  (iii) $\Gamma_{-,0}=\gamma_1\cup \gamma_2$.} 
		\label{figPOmegat} 
		    \end{figure}   

We then present our main result for \eqref{pr:w}, where a precise description of the positive solution set $\{ (\lambda, w_{\lambda})\}$ of \eqref{pr:w} is given if $\delta>\delta_0$ is sufficiently close to $\delta_{0}$, i.e., if the situation of $g_{\delta}$ is assumed to be near the critical case $\int_{\partial\Omega}g_{\delta_0}=0$. 


\begin{theorem} \label{thm2}
Let $g_{\delta}\in C^{1+\theta}(\partial\Omega)$ be introduced by \eqref{gdel}. If $\delta > \delta_0$ is sufficiently close to $\delta_{0}$, then the positive solution set of \eqref{pr:w} with  $g=g_{\delta}$ for $\lambda \in (0,\lambda_1(g_{\delta}))$ is given as follows (Figure \ref{figcurvew}):
\begin{enumerate} \setlength{\itemsep}{0.2cm} 
\item \eqref{pr:w} with $g=g_{\delta}$ possesses a {\rm unique} positive solution $w_{\lambda}$ for every $\lambda\in (0,\lambda_1(g_{\delta}))$, and the positive solution set $\{ (\lambda,w_{\lambda}) :  \lambda\in (0,\lambda_1(g_{\delta})) \}$ is represented by a smooth curve, 
\item $w_{\lambda}\rightarrow 0$ in $C^2(\overline{\Omega})$ 
as $\lambda\nearrow \lambda_1(g_{\delta})$, i.e., bifurcation from $\{ (\lambda,0)\}$ at $(\lambda_1(g_{\delta}),0)$ occurs {\rm subcritically}, and  
\item $w_{\lambda}\rightarrow w_{0}$ in $C^2(\overline{\Omega})$ as $\lambda\searrow 0$ 
for some $w_{0}\in C^2(\overline{\Omega})$, where $w_{0}$ 
is a unique positive solution of \eqref{pr:w} with $g=g_{\delta}$ for $\lambda=0$ (actually the smooth positive solution curve is 
extended slightly to $\lambda<0$). 
\end{enumerate}
\end{theorem}
%
    \begin{figure}[!htb]
    \centering 
    \includegraphics[scale=0.16]{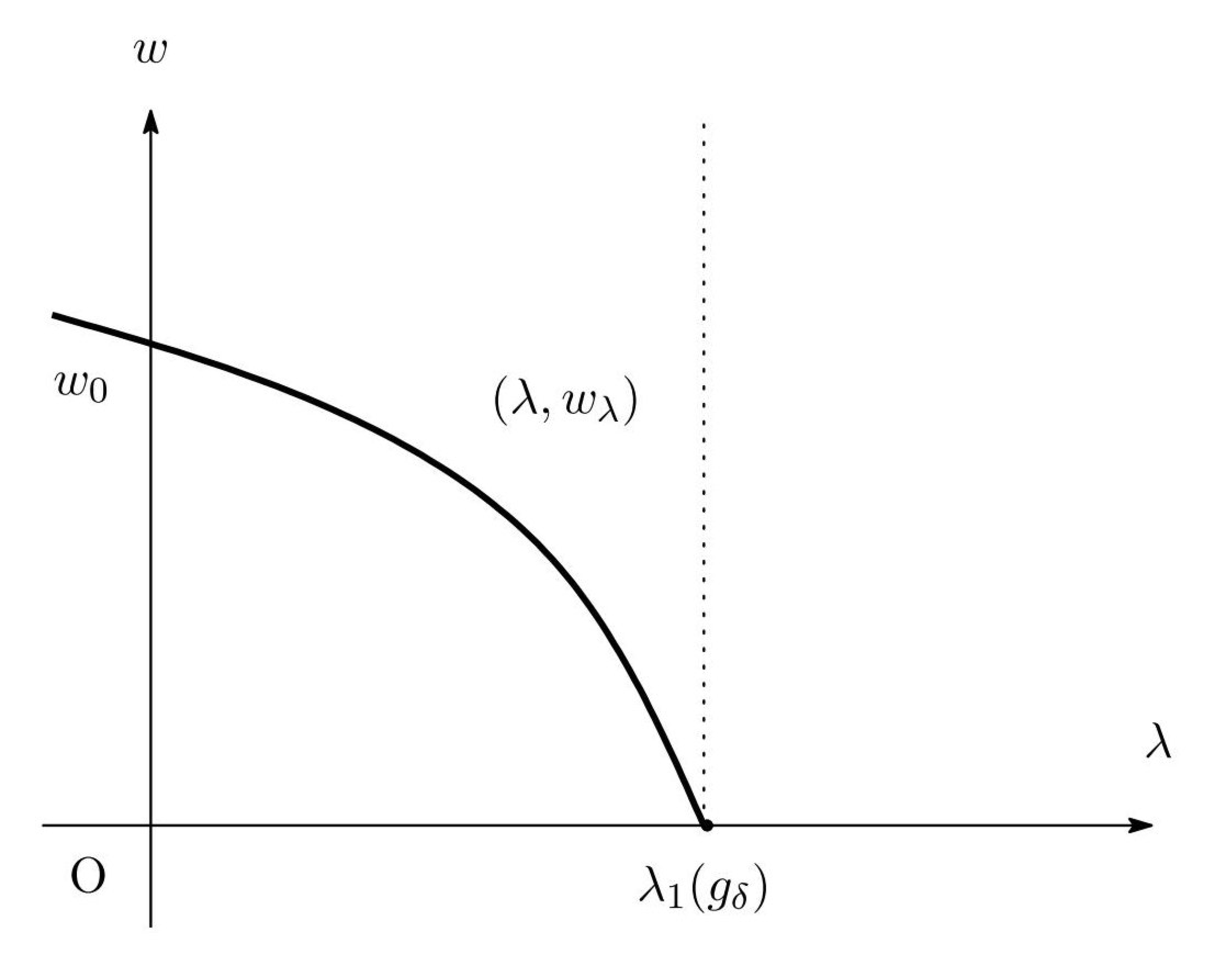} 
		  \caption{Positive solution set $\{ (\lambda,w_\lambda)\}$ of \eqref{pr:w} with $g=g_{\delta}$ as a smooth curve in the case that  $\delta$ is close to $\delta_{0}$.}
		\label{figcurvew} 
    \end{figure} 
%

As a byproduct of assertion (iii), we obtain that if $\delta>\delta_0$ is close to $\delta_{0}$, then the uniqueness of a positive solution holds for \eqref{pr:w} with $g=g_{\delta}$ and $\lambda=0$:
\begin{align*}
\begin{cases}
-\Delta w = 0 & \mbox{ in } \Omega, \\  
\dfrac{\partial w}{\partial \nu}=g(x)w^p & 
\mbox{ on } \partial\Omega. 
\end{cases}    
\end{align*}
We refer to \cite[Theorem 2]{Pf1999} for the existence result for \eqref{pr:w} with $\lambda=0$. 
We also refer to \cite{Bo11} for high multiplicity of positive solutions for  \eqref{Npr} with $\lambda=0$ when the negative part of $m$ is large (see \cite{BGH05} for the Dirichlet case).    

Going back to \eqref{sppr} by \eqref{change}, Theorems \ref{thm1} and \ref{thm2} are analogous to 
Theorems \ref{thm3} and \ref{thm4} in Sections \ref{sec:PT1.1} and \ref{sec:Pthm2}, respectively. 
Theorem \ref{thm4} provides us with a  
precise description of the positive solution set for \eqref{sppr}, 
as shown in Figure \ref{figcurve}.

To conclude the Introduction, it should be noted that our approach employed in the proofs of Theorem \ref{thm1} and \ref{thm2} remains valid for the following more general situation:
\begin{align} \label{pr:w:f}
\begin{cases}
-\Delta w = 0 & \mbox{ in } \Omega, \\  
\dfrac{\partial w}{\partial \nu}=\lambda g(x)w + f(x) w^p & 
\mbox{ on } \partial\Omega, 
\end{cases}
\end{align}
where $p>1$ is subcritical, $\lambda>0$, and $g,f\in C^{1+\theta}(\partial\Omega)$ both change signs (see Theorems \ref{thm:f1} and \ref{thm2:gf} in Sections \ref{sec:PT1.1} and \ref{sec:Pthm2}, respectively). In addition, similar results to Theorems \ref{thm1} and \ref{thm2} can be established for \eqref{Npr}. For high multiplicity of positive solutions to a similar type of \eqref{Npr}, we refer to \cite{Te18}.

The remainder of this paper is organized as follows. 
Section \ref{sec:PT1.1} is devoted to the proof of Theorem \ref{thm1}. 
In Subsection \ref{subsec:ex}, we prove the existence assertion by employing the variational approach based on the Nehari manifold and fibering map with \eqref{pr:w}. 
In Subsection \ref{subsec:none}, 
we use a Picone type identity to verify the nonexistence assertion. 
In Subsection \ref{subsec:insta}, we prove assertion (i) with 
the aid of an idea from \cite{BH90}. 
In Subsection \ref{subsec:BA}, we prove assertions 
(ii) to (iv) using a variational approach.

Section \ref{sec:Pthm2} is devoted to the proof of Theorem \ref{thm2}, where a crucial step is the establishment of Proposition \ref{prop:2nd}. We conduct spectral analysis for the linearized eigenvalue problem at each positive solution of \eqref{pr:w}.

In Section 4, we present an application of Theorems \ref{thm1} and \ref{thm2} to the logistic boundary condition. Our main results in this section are Theorems \ref{thm:A1}, \ref{thm:A2}, and \ref{thm:A3}. In Subsection \ref{subsec:1D}, 
we reduce our analysis to the one-dimensional space case, where a numerical observation is added.


\section{Proof of Theorem \ref{thm1}} 

\label{sec:PT1.1}

In this section, we prove Theorem \ref{thm1}. 

\subsection{Existence} 

\label{subsec:ex}

This subsection assumes that $\int_{\partial\Omega}g<0$. We define the functional associated with \eqref{pr:w}
\begin{align*}
J_\lambda (w) := 
\frac{1}{2}E_\lambda (w) - \frac{1}{p+1}G(w), \quad 
w \in H^1(\Omega), 
\end{align*}
where 
\begin{align*}
E_\lambda (w) := \int_\Omega |\nabla w|^2 - \lambda 
\int_{\partial\Omega} g(x) w^2, \quad 
G(w) := \int_{\partial\Omega} g(x)|w|^{p+1}.  
\end{align*}
Using the divergence theorem, it is easy to deduce that 
\begin{align*} 
\int_{\partial\Omega}g(x)\varphi_1(g)^{2}>0 \quad \mbox{ and } 
\quad G(\varphi_1(g))>0, 
\end{align*}
which are used repeatedly in the following. 

The next lemma implies that $E_\lambda(\cdot)$ is coercive in $H^1(\Omega)$ for $\lambda\in (0,\lambda_1(g))$. Here, $\| \cdot\|$ denotes the usual norm of $H^1(\Omega)$. 
\begin{lem} \label{lem:coer} 
Let $I \subset (0, \lambda_1(g))$ be a compact interval.  Then, 
there exists a constant $c_I > 0$ such that 
\begin{align*}
E_{\lambda} (w) \geq c_I \Vert w \Vert^2, \quad w \in 
H^1(\Omega) \ \mbox{ and } \ \lambda \in I. 
\end{align*}
\end{lem}

For our procedure, 
we use the fibering map for $J_\lambda$, 
$\lambda\in (0, \lambda_1(g))$, which is introduced as follows: 
for $w\neq 0$, we set
\begin{align*}
j_w (t) := J_\lambda (tw) 
= \frac{t^2}{2} E_\lambda (w) - \frac{t^{p+1}}{p+1} G(w), \quad t>0. 
\end{align*}
The associated Nehari manifold is also introduced: 
\begin{align*}
\mathcal{N}_{\lambda} := 
\{ w \in H^1(\Omega)\setminus \{ 0\} : j_{w}^{\prime}(1) = 0 \}
=\{ w\in H^1(\Omega)\setminus \{ 0\}: E_\lambda(w)=G(w) \}. 
\end{align*}
We split $\mathcal{N}_{\lambda}$ into three parts as follows: 
\begin{align*}
\mathcal{N}_{\lambda}^{\pm} &:= \{ w \in \mathcal{N}_{\lambda}:  j_{w}^{\prime\prime}(1)\gtrless 0 \} \\ 
& = \{ w \in H^1(\Omega)\setminus \{ 0\}: E_{\lambda}(w)=G(w), 
E_{\lambda}(w)\gtrless pG(w) \}, \\ 
\mathcal{N}_{\lambda}^{0} &:= 
\{ w \in \mathcal{N}_{\lambda}: j_{w}^{\prime\prime}(1) = 0 \} \\ 
&=\{  w \in H^1(\Omega)\setminus \{ 0\}:   E_{\lambda}(w)=G(w), \ 
E_{\lambda}(w)=pG(w) \}. 
\end{align*}
However, from Lemma \ref{lem:coer}, we deduce that if 
$\lambda\in (0, \lambda_1(g))$, then for any $w\in \mathcal{N}_{\lambda}$, we obtain  
\begin{align*}
E_{\lambda}(w)-pG(w) = (1-p)E_{\lambda}(w)<0. 
\end{align*}
The next lemma is thus proved.
\begin{lem}
$\mathcal{N}_{\lambda}=\mathcal{N}_{\lambda}^{-}\neq 
\emptyset$ and 
$\mathcal{N}_{\lambda}^{+}\cup \mathcal{N}_{\lambda}^{0}=\emptyset$ for $\lambda\in (0,\lambda_1(g))$. 
\end{lem}

\begin{proof}
It remains to prove $\mathcal{N}_{\lambda}^{-}\neq \emptyset$. 
We use a positive eigenfunction $\varphi_1=\varphi_1(g)$ associated with $\lambda_1(g)$ to verify the assertion. We infer from Lemma \ref{lem:coer} that $E_{\lambda}(\varphi_1)>0$. 
Since $G(\varphi_1)>0$, the fibering map $j_{\varphi_1}(\cdot)$ has a unique global maximum point $t_0>0$. This implies that 
$j_{\varphi_1}^{\prime}(t_0) = 0>j_{\varphi_1}^{\prime\prime}(t_0)$. Thus, $t_{0}\varphi_1 \in \mathcal{N}_{\lambda}^{-}$. 
\end{proof}

The next lemma asserts that $J_{\lambda}(\cdot)$ is positive in $\mathcal{N}_{\lambda}^{-}$ for each $\lambda\in (0, \lambda_1(g))$.  
\begin{lem} \label{lem:Jposi}
Let $\lambda \in (0,\lambda_1(g))$. Then, 
$J_{\lambda}(w)>0$ for all $w \in \mathcal{N}_{\lambda}^{-}$. 
\end{lem} 

\begin{proof}
Let $\lambda\in (0,\lambda_1(g))$, and $w\in \mathcal{N}_{\lambda}^{-}$. 
Then, $E_\lambda (w) > 0$ from Lemma  \ref{lem:coer}, which implies that $G(w)>0$. Hence, $J_{\lambda}(w)=j_{w}(1)=\max_{t>0}j_{w}(t)>0$. 
\end{proof}

We then prove the existence of a minimizer in $\mathcal{N}_{\lambda}^{-}$ for $J_{\lambda}$. Let $\lambda \in (0,\lambda_1(g))$. From Lemma \ref{lem:Jposi}, we can say 
$\gamma := \inf_{\mathcal{N}_{\lambda}^{-}} J_{\lambda}(w)\geq 0$, and we take a minimizing sequence $\{ w_n\}\subset  \mathcal{N}_{\lambda}^{-}$ 
such that $J_{\lambda}(w_n) \searrow \gamma$. We then obtain the following lemma.  
\begin{lem} \label{lem:wbdd}
$\sup_{n\geq1} \| w_n \|<\infty$. 
\end{lem}

\begin{proof}
By contradiction, 
we assume $\Vert w_n \Vert \to \infty$. Say $\eta_{n} = \frac{w_{n}}{\Vert w_{n} \Vert}$, and $\Vert \eta_n \Vert = 1$. 
Since $J_{\lambda}(w_{n})$ is bounded by the choice of $w_n$, we infer the existence of $C>0$ such that 
\begin{align*}
\left( \frac{1}{2} - \frac{1}{p+1} \right) E_{\lambda}(w_{n}) 
= J_{\lambda}(w_n)\leq C. 
\end{align*}
This implies that $\varlimsup_{n\to\infty} E_{\lambda}(\eta_n)\leq0$. Thus, by Lemma \ref{lem:coer}, $\eta_n \to 0$ in $H^1(\Omega)$, which is a contradiction. \end{proof}

From Lemma \ref{lem:wbdd}, we may infer that 
for some $w_{0}\in H^1(\Omega)$, $w_{n}\rightharpoonup w_{0}$, and $w_{n}\rightarrow w_{0}$ in $L^2(\Omega)$ and $L^{p+1}(\partial\Omega)$. We then have the following. 
\begin{lem} \label{lem:v0not0} 
\strut 
\begin{enumerate} \setlength{\itemsep}{0.2cm} 
\item $E_{\lambda}(w_{0})>0$ and $G(w_{0})>0$. 
\item $w_{n} \to w_0$ in $H^1(\Omega)$. 
\end{enumerate}
\end{lem} 

\begin{proof}
We verify assertion (i). 
We claim that $\gamma > 0$. Assume by contradiction that 
$J_{\lambda}(w_{n})\rightarrow 0$, and then, from $w_{n}\in \mathcal{N}_{\lambda}^{-}$, we deduce 
\begin{align*}
E_{\lambda}(w_{n})=\left( \frac{1}{2}-\frac{1}{p+1}\right)^{-1}J_{\lambda}(w_{n}) 
\longrightarrow 0. 
\end{align*}
From Lemma \ref{lem:coer}, $\| w_{n} \|\rightarrow 0$. 
However, $\eta_n = \frac{w_n}{\Vert w_n \Vert}$ satisfies 
\begin{align*}
E_{\lambda}(\eta_{n})=G(\eta_{n})\Vert w_n \Vert^{p-1} 
\longrightarrow 0, 
\end{align*}
and using Lemma \ref{lem:coer} again, we obtain  
$\Vert \eta_n \Vert \rightarrow 0$, which is contradictory for $\| \eta_n \|=1$, as desired. Since $w_{n}\in \mathcal{N}_{\lambda}^{-}$ and $\gamma >0$, 
we see 
\begin{align*}
G(w_0)=\lim_{n\to\infty}G(w_n)
= \lim_{n\to\infty}\left( \frac{1}{2}-\frac{1}{p+1}\right)^{-1}J_{\lambda}(w_n)
=\left( \frac{1}{2}-\frac{1}{p+1}\right)^{-1}\gamma>0. 
\end{align*}
In particular, $w_{0}\neq 0$, thus, Lemma \ref{lem:coer} ensures $E_{\lambda}(w_{0})>0$. 
Assertion (i) is now verified. 

We then verify assertion (ii). Assume to the contrary that 
\begin{align} \label{assump:w0}
w_n \not\to w_0 \ \mbox{ in } H^1(\Omega). 
\end{align}
Since $E_{\lambda}(w_{0})>0$ and $G(w_{0})>0$, $j_{w_{0}}$ has a unique global maximum point $t_{0}>0$, and consequently, $t_{0}w_{0}\in \mathcal{N}_{\lambda}^{-}$. By taking a suitable subsequence of $\{ w_{n} \}$, denoted by the same notation, \eqref{assump:w0} infers  
\begin{align*}
J_{\lambda}(t_0 w_0)=
j_{w_0}(t_0)<\lim_{n\to\infty} j_{w_n}(t_0)
\leq \lim_{n\to\infty} j_{w_n}(1)
=\lim_{n\to\infty}J_{\lambda}(w_n)=\gamma, 
\end{align*}
which is contradictory for $t_{0}w_{0}\in \mathcal{N}_{\lambda}^{-}$. Here, we have used the fact that $j_{w_{n}}$ has the unique global maximum point $t=1$. Assertion (ii) has been verified.  
\end{proof}

From Lemma \ref{lem:v0not0}, we derive the following existence result. 
\begin{prop} \label{prop:exist}
Assume that $\int_{\partial\Omega}g<0$. Then, 
\eqref{pr:w} has a variational positive solution for every 
$\lambda\in (0, \lambda_1(g))$. 
\end{prop}

\begin{proof}
By Lemma \ref{lem:v0not0} (i), there exists $t_{0}>0$ 
such that $j_{w_{0}}^{\prime}(t_{0})=0>j_{w_{0}}^{\prime\prime}(t_{0})$, thus, $t_{0}w_{0}\in \mathcal{N}_{\lambda}^{-}$. We claim that $t_{0}=1$. Once this is verified, Lemma \ref{lem:v0not0} (ii) shows that 
\begin{align*}
J_{\lambda}(w_0)=\lim_{n\to\infty}J_{\lambda}(w_n)=\gamma \ \  
\mbox{ and } \ \ w_0\in \mathcal{N}_{\lambda}^{-}, 
\end{align*}
thus, Lemma \ref{lem:Jposi} provides  
\begin{align*}
J_{\lambda}(w_0)=\min_{w\in \mathcal{N}_{\lambda}^{-}}J_{\lambda}(w)>0.     
\end{align*}
Without loss of generality, we may assume that $w_{0}\geq 0$. 
Since $w_0\neq 0$, the strong maximum principle and boundary point lemma apply, and we obtain that $w_{0}>0$ in $\overline{\Omega}$. By a similar argument as in \cite[Theorem 2.3]{BZ03}, 
$w_0$ is a critical point for $J_{\lambda}$. The desired conclusion thus follows.

It remains to show that $t_{0}=1$. Note that if $j_{w_0}^{\prime}(t)=0\geq j_{w_0}^{\prime\prime}(t)$, then $t=t_0$. Since 
$j_{w_n}^{\prime}(1)=0>j_{w_n}^{\prime\prime}(1)$ from 
$w_n\in \mathcal{N}_{\lambda}^{-}$, passing to the limit provides 
that $j_{w_0}^{\prime}(1)=0\geq j_{w_0}^{\prime\prime}(1)$, using 
Lemma \ref{lem:v0not0} (ii). Thus, $t_0=1$. 
\end{proof}

\subsection{Nonexistence}

\label{subsec:none}

This subsection deals with the general case of $g$ such that $g\in C^{1+\theta}(\partial\Omega)$ changes sign. 
For $\lambda\in \mathbb{R}$, let 
$\sigma_1=\sigma_1(\lambda)\in \mathbb{R}$ be the smallest eigenvalue of the eigenvalue problem 
\begin{align*}
\begin{cases}
-\Delta \phi = \sigma \phi & \mbox{ in } \Omega, \\
\dfrac{\partial \phi}{\partial \nu} = \lambda g(x) \phi
& \mbox{ on } \partial\Omega. 
\end{cases}
\end{align*}
We can characterize $\sigma_1$ by the following variational formula (\cite{Um2006}): 
\begin{align} \label{epr:sig}
\sigma_1= 
\inf\left\{ \int_{\Omega}|\nabla \phi|^2 - \lambda \int_{\partial\Omega}g(x)\phi^2 : \int_{\Omega}\phi^2=1 \right\}. 
\end{align}
We know that $\sigma_1$ is simple and has positive eigenfunctions, and it satisfies 
\begin{align} \label{sig1sign}
\left\{\begin{array}{rl}
\sigma_1(0)=\sigma_1(\lambda_1(g))=0, & \\
\sigma_1(\lambda)>0 & \mbox{ for } \ \lambda\in (0, \, \lambda_1(g)), \\ 
\sigma_1(\lambda)<0 & \mbox{ for } \ \lambda>\lambda_1(g). 
\end{array} \right. 
\end{align}
We call $\Phi_1 \in C^{2+\alpha}(\overline{\Omega})$ a positive eigenfunction of \eqref{epr:sig} that is associated with $\sigma_1$, satisfying that $\Phi_1>0$ in $\overline{\Omega}$. Note that $\Phi_1$ coincides with $\varphi_1(g)$ when $\lambda=\lambda_1(g)$. 

We then obtain the following nonexistence result. 
\begin{prop} \label{prop:non}
Let $p>1$ be arbitrary {\rm(}without \eqref{psubcr}{\rm)}. Then, 
\eqref{pr:w} has no positive solution for any $\lambda\geq\lambda_1(g)$. 
\end{prop}

\begin{proof}
First, we consider the case when $\int_{\partial\Omega}g\geq0$ 
and $\lambda=0$. Recall that $\lambda_1(g)=0$ in this case. 
If $w$ is a positive solution of \eqref{pr:w} for $\lambda=0$, then $w$ is not a constant, and the divergence theorem gives 
\begin{align*}
0=\int_{\Omega} -\Delta w \, w^{-p} = \int_{\Omega} \nabla w 
\nabla (w^{-p}) -\int_{\partial\Omega}g < -\int_{\partial\Omega}g \leq0.
\end{align*}
The assertion thus follows. 

Next, we consider the case of $\lambda>0$, and prove the nonexistence assertion only for the case of $\int_{\partial\Omega}g<0$. The case of  $\int_{\partial\Omega}g\geq0$ is similar, so we omit the proof. Let $w$ be a positive solution of \eqref{pr:w} for $\lambda>0$. We then use the Picone type identity
\begin{align*}
&-\sigma_1\Phi_1^{p+1}w^{1-p} \\ 
&= \left( \frac{\Phi_1}{w}\right)^p (\Delta \Phi_1 w - \Phi_1 \Delta w) = \left( \frac{\Phi_1}{w} \right)^p \sum_{j=1}^{N} \frac{\partial}{\partial x_j}\left( w^2\frac{\partial}{\partial x_j} \left( \frac{\Phi_1}{w} \right)\right). 
\end{align*}
The divergence theorem shows that 
\begin{align*}
& -\sigma_1 \int_{\Omega} \Phi_1^{p+1}w^{1-p} \\ 
&= \int_{\Omega}\left( \frac{\Phi_1}{w} \right)^p \sum_{j=1}^{N}\frac{\partial}{\partial x_j} \left( w^2\frac{\partial}{\partial x_j} \left( \frac{\Phi_1}{w} \right) \right) \\
& = -p\int_\Omega \left(\frac{\Phi_1}{w}\right)^{p-1}w^{2}  \left\vert \nabla \left( \frac{\Phi_1}{w} \right) \right\vert^{2} + 
\int_{\partial\Omega} \left( \frac{\Phi_1}{w} \right)^{p} 
\left( \frac{\partial \Phi_1}{\partial \nu}w-
\Phi_1 \frac{\partial w}{\partial \nu} \right) \\
&= -p\int_\Omega \left(\frac{\Phi_1}{w}\right)^{p-1}w^{2}  \left\vert \nabla \left( \frac{\Phi_1}{w} \right) \right\vert^{2} -\int_{\partial\Omega} g\Phi_1^{p+1} \\ 
& \leq -\int_{\partial\Omega} g\Phi_1^{p+1} \\ 
& = \lambda^{-1}\left( \sigma_1 \int_{\Omega} \Phi_1^{p+1} 
- p \int_{\Omega} \Phi_1^{p-1}|\nabla \Phi_1|^2 \right). 
\end{align*}
Here, the last equality is deduced from the computation of $\int_{\Omega} -\Delta \Phi_1 \Phi_1^{p}$ via the divergence theorem. Thus, using \eqref{sig1sign}, we obtain that $\lambda\in (0,\lambda_1(g))$. \end{proof}

\subsection{Instability} 

\label{subsec:insta}

Let $w$ be a positive solution of \eqref{pr:w} for $\lambda\geq0$. 
We call $\gamma_1(\lambda,w) \in \mathbb{R}$ 
the smallest eigenvalue of the eigenvalue problem 
\begin{align} \label{epro27apr}
\begin{cases}
-\Delta \phi = \gamma \phi & \mbox{ in } \Omega, \\
\dfrac{\partial \phi}{\partial \nu}=\lambda g(x) \phi + pg(x)w^{p-1}\phi + \gamma \phi & \mbox{ on }  \partial\Omega. 
\end{cases}
\end{align}
It is well known that $\gamma_1$ is simple, and the corresponding eigenfunctions possess constant signs. For a nonnegative  eigenfunction $\phi_1(\lambda,w)$ associated with $\gamma_1$, the strong maximum principle and boundary point lemma shows that $\phi_1>0$ in $\overline{\Omega}$. The positive solution $w$ is called {\it asymptotically stable} and {\it  unstable} if $\gamma_1>0$ and $\gamma_1<0$, respectively. 

In 
view of Propositions \ref{prop:exist}, we then prove the following instability result. 

\begin{prop} \label{lem:inst}
Let $p>1$ be arbitrary {\rm(}without \eqref{psubcr}{\rm)}. 
If $\int_{\partial\Omega}g<0$, then every positive solution of \eqref{pr:w} for $\lambda\in [0,\lambda_1(g))$ is unstable.  
\end{prop}

\begin{proof}
The proof is in the same spirit of that for \cite[Theorem 1]{BH90}. 
Let $w$ be a positive solution of \eqref{pr:w} for $\lambda\in [0,\lambda_1(g))$. 
Say $h(t):=\lambda t + t^p$, $t>0$, and then, for any $\lambda>0$, we see that $h(w)>0$ in $\overline{\Omega}$, and $h''(w)=p(p-1)w^{p-2}>0$ in $\overline{\Omega}$. We also see that $\gamma_1$ and $\phi_1$ satisfy  
\begin{align*}
\begin{cases}
-\Delta \phi_1 = \gamma_1 \phi_1 & \mbox{ in } \Omega, \\
\dfrac{\partial \phi_1}{\partial \nu}=g(x) h'(w) \phi_1 + 
\gamma_1 \phi_1 & \mbox{ on }  \partial\Omega. 
\end{cases}    
\end{align*}
The divergence theorem shows that 
\begin{align*}
\gamma_1 \int_{\Omega} h(w)\phi_1 
&=\int_{\Omega} (-\Delta \phi_1 h(w) + \Delta w h'(w) \phi_1) \\
&= -\int_{\Omega} h''(w)|\nabla w|^2 \phi_1 - 
\gamma_1 \int_{\partial\Omega} h(w) \phi_1. 
\end{align*}
Therefore, 
\begin{align*}
\gamma_1 = -\frac{\int_{\Omega}h''(w)|\nabla w|^2 \phi_1}{\int_{\Omega} h(w)\phi_1 + \int_{\partial\Omega}h(w)\phi_1}<0, 
\end{align*}
as desired. 
\end{proof}

\subsection{Bounds and Asymptotics}

\label{subsec:BA}

We complete the proof of Theorem \ref{thm1} in this subsection. 
The next proposition asserts that the variational positive solution 
$w_{\lambda}$ is bounded in $H^1(\Omega)$ for $\lambda\in (0, \lambda_1(g))$.

\begin{prop} \label{lem:bdd9mar21} 
Assume that $\int_{\partial\Omega}g<0$. 
Let $w_{\lambda}$ be the variational positive solution of \eqref{pr:w} for $\lambda\in (0,\lambda_1(g))$, constructed by Proposition \ref{prop:exist}. 
Then, 
\[
\sup_{\lambda\in (0, \lambda_1(g))} \| 
w_\lambda \|<\infty. 
\]
\end{prop} 


\begin{proof}
We verify the existence of $C>0$ satisfying that 
if $w_{n}:=w_{\lambda_{n}}$ is the positive solution of \eqref{pr:w} for $\lambda=\lambda_{n}\in (0,\lambda_1(g))$, then 
\begin{align} \label{Jcla}
J_{\lambda_n}(w_n) = \inf_{w\in \mathcal{N}_{\lambda_{n}}^{-}} J_{\lambda_{n}}(w)\leq C.
\end{align}
To this end, we take a smooth function $\hat{w}$ in $\overline{\Omega}$  such that 
$G(\hat{w})>0$. Then, Lemma \ref{lem:coer} gives $E_{\lambda_n}(\hat{w})>0$. 
Therefore, the mapping 
$$
t\mapsto \frac{t^2}{2} E_{\lambda_n}(\hat{w})-\frac{t^{p+1}}{p+1}G(\hat{w})
$$
has a global maximum point $t=t_n>0$, implying that $t_n\hat{w}\in \mathcal{N}_{\lambda_n}^{-}$.  Thus, 
$J_{\lambda_n}(w_n)\leq J_{\lambda_n}(t_n\hat{w})$. 
Observing now that 
\begin{align*}
J_{\lambda_n}(t_n \hat{w}) 
& \leq \frac{t_{n}^{2}}{2} \left( \int_\Omega |\nabla \hat{w}|^2 + \lambda_1(g)\int_{\partial\Omega} g^{-}(x) \hat{w}^2 \right)-\frac{t_{n}^{p+1}}{p+1} G(\hat{w}) \\ 
& \leq \sup_{t>0} \left\{ \frac{t^2}{2} \left( \int_\Omega |\nabla \hat{w}|^2 + \lambda_1(g) \int_{\partial\Omega}g^{-}(x)\hat{w}^2 \right) - \frac{t^{p+1}}{p+1} G(\hat{w}) \right\}
< \infty, 
\end{align*}
assertion \eqref{Jcla} is thus verified.

We then prove this proposition. Assume by contradiction that $w_n:=w_{\lambda_n}$ is the positive solution of \eqref{pr:w} for $\lambda=\lambda_n\in (0,\lambda_1(g))$ with the condition that 
$\lambda_n\to \lambda_\infty\in [0,\lambda_1(g)]$ and $\| w_n\|\rightarrow \infty$. Letting $\eta_n = \frac{w_n}{\Vert w_n \Vert}$, we may deduce that for some $\eta_0\in H^1(\Omega)$, 
$\eta_n \rightharpoonup \eta_0$, and $\eta_{n}\rightarrow \eta_{0}$ in $L^2(\Omega)$ and $L^{p+1}(\partial\Omega)$. From $w_n \in \mathcal{N}_{\lambda_n}^{-}$, we see that 
\begin{align} \label{EJn}
E_{\lambda_n}(w_n)
=\left( \frac{1}{2}-\frac{1}{p+1}\right)^{-1}
J_{\lambda_n}(w_n).  
\end{align}
Combining \eqref{Jcla} and \eqref{EJn} shows that $\varlimsup_{n\to\infty}E_{\lambda_n}(\eta_n)\leq0$. Because $\| \eta_{n}\|=1$, 
Lemma \ref{lem:coer} shows that $\lambda_\infty=0$ or $\lambda_1(g)$, and 
\begin{align*}
0\leq E_{\lambda_\infty}(\eta_0) 
\leq \varliminf_{n\to\infty} E_{\lambda_n}(\eta_n) 
\leq \varlimsup_{n\to\infty} E_{\lambda_n}(\eta_n) \leq 0.  
\end{align*}
This implies that $\eta_n \rightarrow \eta_0$ in $H^1(\Omega)$,  $\|\eta_0 \|=1$, and $\eta_0\geq0$. Moreover, if $\lambda_\infty=0$, then $\eta_0$ is a positive constant, 
whereas $\eta_0=c\varphi_1(g)$ for some $c>0$ if $\lambda_\infty=\lambda_1(g)$. 
However, since $w_n\in \mathcal{N}_{\lambda_n}^{-}$, we infer 
\begin{align*}
G(\eta_0)=\lim_{n\to\infty}G(\eta_n)=\lim_{n\to\infty}E_{\lambda_n}(\eta_n)
\| w_n \|^{1-p} = 0, 
\end{align*}
thus, 
\begin{align*}
0 = G(\eta_0) = \left\{ 
\begin{array}{ll}
\eta_0^{p+1} \int_{\partial\Omega}g(x)<0, & \lambda_\infty=0, \\
c^{p+1} G(\varphi_1(g))>0, & \lambda_\infty=\lambda_1(g), 
\end{array} \right. 
\end{align*}
which is a contradiction. 
\end{proof}

\begin{cor} \label{cor:bif}
Under the conditions of Proposition \ref{lem:bdd9mar21}, 
$w_\lambda \rightarrow 0$ in $C^2(\overline{\Omega})$ as 
$\lambda\nearrow \lambda_1(g)$.  
\end{cor}

\begin{proof}
By elliptic regularity (\cite[Theorem 2.2]{Ro2005}), 
it follows from Proposition \ref{lem:bdd9mar21} that 
$\sup_{\lambda\in (0,\lambda_1(g))}\| w_{\lambda}  \|_{W^{1,r}(\Omega)}<\infty$ 
for $r>N$. 
Combining the H\"older estimate and the compactness argument show that, up to a subsequence, $w_{\lambda_n} \rightarrow w_\infty\geq0$ in $C^2(\overline{\Omega})$ as $\lambda_n\nearrow \lambda_1(g)$, and $w_\infty$ is a nonnegative solution of \eqref{pr:w} for $\lambda=\lambda_1(g)$. Proposition \ref{prop:non} completes the proof with $w_\infty=0$. 
\end{proof}

The following {\it a priori} lower bound can be derived for the positive solutions of \eqref{pr:w}:  
\begin{prop} \label{prop:bbelow}
Assume $\int_{\partial\Omega}g<0$. Then, for $\overline{\lambda}\in (0, \lambda_1(g))$, 
there exists $C>0$ such that if $w$ is a positive solution of \eqref{pr:w} for $\lambda\in (0,\overline{\lambda})$, then 
$\| w \|_{C(\overline{\Omega})}>C$. 
\end{prop}

\begin{proof}
If not, then we may take a positive solution $w_n$ of \eqref{pr:w} for $\lambda=\lambda_n\searrow 0$ such that $\| w_n \|_{C(\overline{\Omega})}\rightarrow 0$ because we have no bifurcation of positive solutions from $\{ (\lambda,0) : 0<\lambda < \lambda_1(g) \}$ for \eqref{pr:w}. 
Then, $\| w_n\|\rightarrow 0$, and additionally, Lemma \ref{lem:coer} shows that $0<E_{\lambda_n}(w_n)=G(w_n)$. 
Say $\eta_n=\frac{w_n}{\| w_n\|}$, and we may deduce that for some $\eta_\infty \in H^1(\Omega)$, $\eta_n \rightharpoonup \eta_\infty$, and $\eta_{n} \rightarrow \eta_{\infty}$ 
in $L^2(\Omega)$ and $L^{p+1}(\partial\Omega)$. Then, $\eta_n$ admits 
\begin{align*}
& \int_{\Omega}|\nabla \eta_n|^2=\lambda_n \int_{\partial\Omega}g(x)\eta_n^2+ \| w_n \|^{p-1} \int_{\partial\Omega}g(x)\eta_n^{p+1} \longrightarrow 0, \\ 
& G(\eta_n) \longrightarrow \int_{\partial\Omega}g(x)\eta_\infty^{p+1}\geq0. 
\end{align*}
The first assertion implies $\eta_n \rightarrow \eta_\infty$ 
in $H^1(\Omega)$, and $\eta_\infty$ is a positive constant. 
Thus, $\int_{\partial\Omega}g(x)\geq0$, which is a contradiction. 
\end{proof}

\noindent{\it End of proof of Theorem \ref{thm1}.} 
The existence and nonexistence assertions follow from Propositions \ref{prop:exist} and \ref{prop:non}, respectively. 
Proposition \ref{lem:inst} implies assertion (i). 
Assertion (ii) is verified by Corollary \ref{cor:bif}. 
Assertion (iii) is obtained from Proposition \ref{lem:bdd9mar21} by the bootstrap argument as in the proof of Corollary \ref{cor:bif}. 
Finally, Proposition \ref{prop:bbelow} implies assertion (iv).   \qed

Recalling \eqref{change}, we convert Theorem \ref{thm1} to \eqref{sppr} 
with $v_\lambda=\lambda^{-\frac{1}{p-1}}w_{\lambda}$ to formulate the following theorem:  

\begin{theorem} \label{thm3}
Assume that $\int_{\partial\Omega}g(x)<0$. Then, 
\eqref{sppr} possesses a positive solution $v_\lambda$ 
for $\lambda\in (0,\lambda_1(g))$ and no 
positive solution for any $\lambda\geq \lambda_1(g)$. 
Moreover, $v_{\lambda}$ satisfies the following {\rm four} conditions:

\begin{enumerate} \setlength{\itemsep}{0.2cm} 
\item $v_\lambda$ is unstable. 
\item $\| v_\lambda \|_{C^{2}(\overline{\Omega})}\rightarrow 0$ 
as $\lambda\nearrow \lambda_1(g)$.  
\item $\sup_{\lambda\in (\underline{\lambda},\lambda_1(g))}\| v_\lambda \|_{C(\overline{\Omega})} <\infty$ for any $\underline{\lambda}\in (0,\lambda_1(g))$, and additionally, the following asymptotics holds:  
\begin{align*} 
c_1\lambda^{-\frac{1}{p-1}}\leq \| v_\lambda \|_{C(\overline{\Omega})}\leq c_2\lambda^{-\frac{1}{p-1}} \quad 
\mbox{ as } \lambda\searrow 0 
\end{align*}
for some $c_2>c_1>0$. 
\item $\inf_{\lambda\in (0, \overline{\lambda})}\| v_\lambda \|_{C(\overline{\Omega})}>0$ for any $\overline{\lambda}\in (0,\lambda_1(g))$. 
\end{enumerate}
\end{theorem}

Similar results to those of Theorem \ref{thm1} can be presented for \eqref{pr:w:f}, which is evaluated in the same way as employed in this section:

\begin{theorem} \label{thm:f1}
Assume that $\int_{\partial\Omega}g(x)<0$ and $\int_{\partial\Omega}f(x)\varphi_1(g)^{p+1}>0$. Then, \eqref{pr:w:f} has a positive solution $W_\lambda$ for 
every $\lambda\in (0,\lambda_1(g))$ and no positive solution for $\lambda=\lambda_1(g)$, satisfying 
\begin{enumerate} \setlength{\itemsep}{0.2cm} 
\item $\sup_{\lambda\in (\underline{\lambda},\lambda_1(g))}\| W_{\lambda} \|_{C(\overline{\Omega})}<\infty$ for any $\underline{\lambda}\in (0, \lambda_1(g))$ and 
\item  $\| W_{\lambda}\|_{C^2(\overline{\Omega})}\rightarrow 0$ as $\lambda\nearrow \lambda_1(g)$. 
\end{enumerate}
If additionally $\int_{\partial\Omega}f\neq 0$, then it is possible to take $\underline{\lambda}=0$. Moreover, 
\begin{enumerate}\setlength{\itemsep}{0.2cm} 
\item[(iii)] $\int_{\partial\Omega}f<0$ implies that $\inf_{\lambda\in (0,\overline{\lambda})}\| W_{\lambda}\|_{C(\overline{\Omega})}>0$ for 
$\overline{\lambda}\in (0,\lambda_1(g))$, whereas 
\item[(iv)] $\int_{\partial\Omega}f>0$ implies that $\inf_{\lambda\in (\underline{\lambda},\overline{\lambda})}\| W_{\lambda}\|_{C(\overline{\Omega})}>0$ for $0< \underline{\lambda}<\overline{\lambda}<\lambda_1(g)$, and  
$\| W_{\lambda}\|_{C^2(\overline{\Omega})} \rightarrow 0$ 
as $\lambda \searrow 0$. 
\end{enumerate}
\end{theorem}

\begin{rem}
\strut 
\begin{enumerate} \setlength{\itemsep}{0.2cm} 

\item Assumption $\int_{\partial\Omega}f(x)\varphi_1(g)^{p+1}>0$ corresponds to $G(\varphi_1(g))>0$ for \eqref{pr:w}. 

\item The case of $\int_{\partial\Omega}f=0$ must be treated delicately. For instance, let $f$ satisfy \eqref{G} additionally. Then, we obtain assertion (i) with $\underline{\lambda}=0$, 
which is derived from the argument in the first paragraph of the proof of Lemma \ref{lem:vaw} below. Consequently, we obtain the conclusion of (iv) because \eqref{pr:w:f} has no positive solutions for $\lambda=0$ when $\int_{\partial\Omega}f\geq 0$. In case $\int_{\partial\Omega}f=0$, how $W_{\lambda}$ behaves as $\lambda \searrow 0$ is an open question. We cannot yet exclude the possibility that $\| W_{\lambda}\|_{C(\overline{\Omega})}\rightarrow \infty$ as $\lambda \searrow 0$.

\end{enumerate}
\end{rem}

\section{Proof of Theorem \ref{thm2}}

\label{sec:Pthm2}

In this section, we consider \eqref{pr:w} with $g=g_{\delta}$ introduced by \eqref{gdel}, and prove Theorem \ref{thm2}. Note that 
$\int_{\partial\Omega}g_{\delta}<0$. 

\subsection{Vanishing positive solutions}

\label{subsec:vpsol}

Let $g_{\delta}\in C^{1+\theta}(\partial\Omega)$, $\delta>\delta_{0}$, satisfy \eqref{gdel}. We then deduce 
\begin{align} \label{lam1del0}
\lambda_1(g_{\delta}) \longrightarrow 0 \ \mbox{ as } \ 
\delta\searrow \delta_{0}. 
\end{align}
Indeed, \eqref{lam1del0} is verified in a similar manner as in the proof of \cite[Lemma 6.6]{RQU2017}, using the condition that  $\int_{\partial\Omega}g_{\delta}\nearrow \int_{\partial\Omega}g_{\delta_{0}}=0$ as $\delta\searrow \delta_{0}$. 
Letting $\varphi_1(g_{\delta})$ be the positive eigenfunction associated 
with $\lambda_1(g_{\delta})$ such that $\| \varphi_1(g_{\delta})\|=1$, we additionally derive from \eqref{lam1del0} that 
\begin{align*}
\int_{\Omega}|\nabla \varphi_1(g_{\delta})|^2 = \lambda_1(g_{\delta})\int_{\partial\Omega}g_{\delta}(x)\varphi_1(g_{\delta})^2 
\longrightarrow 0 \quad \mbox{ as } \ \delta\searrow \delta_0. 
\end{align*} 
This implies that $\varphi_1(g_{\delta})$ converges to a constant $c>0$ in $H^1(\Omega)$ as $\delta\searrow \delta_0$, and by elliptic regularity, 
\begin{align} \label{vphi1gotoconst}
\varphi_1(g_{\delta}) \longrightarrow c \quad \mbox{ in } \ C(\overline{\Omega}) 
\quad \mbox{ as } \ \delta \searrow \delta_0. 
\end{align}

The next lemma is crucial (cf.\ Proposition \ref{prop:2nd}), which asserts that a positive solution of \eqref{pr:w} with $g=g_{\delta}$ shrinks to zero in $\lambda\in [0,\lambda_1(g_{\delta}))$ as $\delta\searrow \delta_{0}$. 

%
\begin{lem} \label{lem:vaw}
Let 
\begin{align*}
C_{\delta}:= 
&\sup\{ \| w \|_{C(\overline{\Omega})}: \mbox{$w$ is a positive solution of \eqref{pr:w}} \\ 
&\hspace*{4.5cm}\mbox{with $g=g_{\delta}$ for $\lambda\in [0,\lambda_1(g_{\delta}))$} \}. 
\end{align*}
Then, $C_{\delta}\rightarrow 0$ as $\delta\searrow \delta_{0}$.
\end{lem}

\begin{proof}
Let $\overline{\delta}>\delta_{0}$ be fixed as close to $\delta_{0}$. 
First, we prove 
\begin{align}\label{ubddw} 
\sup_{\delta\in (\delta_{0},\overline{\delta})}C_{\delta} 
<\infty.    
\end{align}
If not, then there exist 
$\delta_n\in (\delta_{0},\overline{\delta})$ and 
positive solutions $w_n$ of \eqref{pr:w} 
with $g=g_{\delta_n}$ for $\lambda_n \in  [0,\lambda_1(g_{\delta_n}))$
such that $\| w_n \|_{C(\overline{\Omega})}\rightarrow \infty$. Since $w_n$ is harmonic in $\Omega$, and $g_{\delta_n}=g^{+}-\delta_n g^{-}$, we see  
\begin{align} \label{xnmax}
\| w_n \|_{C(\overline{\Omega})}=
\max_{\partial\Omega}w_n
= w_n(x_n), \quad x_n\in \Gamma_{+}. 
\end{align}
Here, we have used the fact that $\frac{\partial w_n}{\partial \nu}(x_n)>0$ by applying the strong maximum principle and boundary point lemma. 
Since $\Gamma_{+}$ is compact, up to a subsequence, 
$x_n \rightarrow x_\infty$ for some $x_\infty\in \Gamma_{+}$, 
as well as $x_{\infty}$ is an interior point of $\Gamma_{+}$. 
This leads us to a contradiction 
using the blow up argument as in the proof of Lemma 3.4 by Kim, Liang, and Shi \cite[Sect.\ 5]{KLS2015}. Assertion \eqref{ubddw} is thus proved. 
Immediately, we have \eqref{ubddw} with $\|\cdot  \|_{C(\overline{\Omega})}$ replaced by $\| \cdot\|$ because it is seen from \eqref{lam1del0} that $\lambda_1(g_\delta)$ is bounded above.

We then consider $\delta\searrow \delta_{0}$, and take a positive solution $w$ of \eqref{pr:w} with $g=g_{\delta}$ for $\lambda\in [0,\lambda_1(g_{\delta}))$. The boundedness of $w$ in $H^1(\Omega)$ infers that there exists $\delta_n\searrow \delta_{0}$ and $\lambda_n\in [0,\lambda_1(g_{\delta_n}))$ such that for some $w_{0}\in H^1(\Omega)$, $w_{n}\rightharpoonup w_{0}$, and $w_{n}\rightarrow w_{0}$ in $L^2(\Omega)$ and $L^{p+1}(\partial\Omega)$. We claim that $w_n \rightarrow 0$ in $H^1(\Omega)$. From the definition of $w_n$, we deduce 
\begin{align*}
\int_{\Omega}\nabla w_n \nabla \psi = \lambda_n  
\int_{\partial\Omega}g_{\delta_n}(x) w_n \psi + \int_{\partial\Omega} g_{\delta_n}(x) w_n^{p}\psi, \quad 
\psi \in H^1(\Omega), 
\end{align*}
thus, because of \eqref{lam1del0}, passing to the limit yields
\begin{align*}
\int_{\Omega}\nabla w_{0}\nabla \psi = 
\int_{\partial\Omega} g_{\delta_{0}}(x) w_{0}^{p}\psi. 
\end{align*}
This implies that $w_{0}$ is a nonnegative weak solution of \eqref{pr:w} with $g=g_{\delta_{0}}$ for  $\lambda=0$. Therefore, $w_{0}\equiv 0$ because  $\int_{\partial\Omega}g_{\delta_{0}}=0$; thus, 
$w_n\rightarrow 0$ in $L^{p+1}(\partial\Omega)$. 
Finally, we deduce 
\begin{align*}
\int_{\Omega}|\nabla w_n|^2=\lambda_n \int_{\partial\Omega}g_{\delta_{0}}w_n^2 + \int_{\partial\Omega}g_{\delta_n}(x)w_n^{p+1}\longrightarrow 0, 
\end{align*}
as desired.

By using the bootstrap argument with elliptic regularity, as in the proof of Corollary \ref{cor:bif}, we then deduce that $\|w_n\|_{W^{1,r}(\Omega)}\rightarrow 0$ for $r>N$. Sobolev's embedding theorem shows that $\| w_n\|_{C(\overline{\Omega})}\rightarrow 0$, as desired.  \end{proof}

\begin{rem} \label{rem:uniq}
Let $\delta>\delta_0$ be fixed. 
A positive solution of \eqref{pr:w} with $g=g_{\delta}$ is unique for $\lambda < \lambda_1(g_\delta)$ that is close to $\lambda_1(g_{\delta})$. 
This is verified by the combination of Proposition \ref{prop:non}, the fact that the upper bound $C_{\delta}$ introduced by Lemma \ref{lem:vaw} is finite for a fixed  $\delta>\delta_0$, with the existence of a unique local bifurcation curve of positive solutions from the simple eigenvalue $\lambda_1(g_{\delta})$. 
This assertion holds unconditionally for $\delta>\delta_0$. 
Indeed, let $w_n$ be a positive solution of \eqref{pr:w} with $g=g_\delta$ as $\lambda_n \nearrow \lambda_1(g_\delta)$. 
Similarly as in the proof of Corollary \ref{cor:bif}, we obtain that up to a subsequence, $w_n \rightarrow 0$ in $C^2(\overline{\Omega})$. Thus, the uniqueness is deduced. 
\end{rem}

\subsection{Our strategy} 

Let $w$ be a positive solution of \eqref{pr:w} with $g=g_{\delta}$ 
for $\lambda \in [0, \lambda_1(g_{\delta}))$. Proposition \ref{lem:inst} 
tells us that $w$ is unstable, meaning that the smallest eigenvalue $\gamma_1$ of \eqref{epro27apr} with $g=g_{\delta}$ for $(\lambda,w)$ is negative. We then look for a certain condition of $\delta$ under which it does {\it not} have a zero eigenvalue. More precisely, for the application of the implicit function theorem to all the $(\lambda,w)$,
we find such $\delta$ that is uniform in $(\lambda,w)$, based on the Fredholm alternative. Thus, we deduce assertions (i) to (iii) in Theorem \ref{thm2}, using the existence assertion of Theorem \ref{thm1} and the uniqueness assertion of Remark \ref{rem:uniq}. 

We then present our strategy more precisely in the general setting of \eqref{pr:w}. We consider 
the eigenvalue problem 
\begin{align} \label{mu:epr}
\begin{cases}
-\Delta \phi = 0 & \mbox{ in } \Omega, \\
\dfrac{\partial \phi}{\partial \nu}=\lambda g(x) 
\phi + \mu g(x) w^{p-1} \phi & \mbox{ on } \partial\Omega, 
\end{cases}
\end{align}
where $w$ is a positive solution of \eqref{pr:w} for $\lambda\geq 0$. We deduce that \eqref{epro27apr} has a zero eigenvalue if and only if $\mu=p$ is an eigenvalue of \eqref{mu:epr}. Therefore, let us study the distribution of the eigenvalues of \eqref{mu:epr}. As 
in \eqref{epro01}, an eigenvalue of \eqref{mu:epr} is called principal if the eigenfunctions associated with it have constant sign. We observe that \eqref{mu:epr} possesses exactly two principal eigenvalues $\mu_{1}^{\pm}\in \mathbb{R}$, which are both simple, satisfying 
\begin{align*}
\left\{ 
\begin{array}{ll}
\mu_1^{-}=0<1=\mu_1^{+},  &  \lambda=0, \\
\mu_1^{-}<0<1=\mu_1^{+},  &  \lambda\in (0, \lambda_1(g)).
\end{array} \right. 
\end{align*}
Indeed, $\mu=1$ is always a principal eigenvalue of \eqref{mu:epr} with the positive eigenfunction $\phi=w$ if $\lambda\in [0, \lambda_1(g))$. It is clear that $\mu_1^{-}=0$ if $\lambda=0$. Lemma \ref{lem:coer} provides that $\mu_{1}^{-}<0$ for $\lambda \in (0, \lambda_1(g))$. 
Moreover, by a similar argument as in \cite{BL80}, we obtain that an eigenvalue $\mu$ larger than $\mu_{1}^{+}$ is not principal. From these facts, we deduce that if $p<\mu_2^{+}$ for a second positive eigenvalue $\mu_{2}^{+}$ of \eqref{mu:epr}, then \eqref{epro27apr} has no zero eigenvalue, which is our desired situation (the implicit function theorem is applicable to $(\lambda, w)$).

\subsection{Analysis of the second eigenvalue} 

\label{subsec:second} 

We complete the proof of Theorem \ref{thm2} by proving the following result.

\begin{prop} \label{prop:2nd}
Let $g_{\delta}\in C^{1+\theta}(\partial\Omega)$ be introduced by 
\eqref{gdel}. For $\delta>\delta_{0}$, let $m_{\delta}\geq1$ be given by the formula 
\begin{align*}
m_{\delta}
= & \inf\{ \mu_{2}^{+} : \mbox{ $\mu_{2}^{+}$ is the second positive eigenvalue of \eqref{mu:epr} } \\
&\mbox{ for a positive solution $w$ of \eqref{pr:w} 
with $g=g_{\delta}$ for $\lambda\in [0,\lambda_1(g_{\delta}))$} \}. 
\end{align*}
Then, $m_{\delta} \rightarrow \infty$ as $\delta\searrow \delta_{0}$. 
\end{prop}

\begin{proof}
We call $\phi_{2}$ the eigenfunction associated with $\mu_{2}^{+}$ and satisfying $\| \phi_{2} \|=1$, and $\phi_{2}$ changes sign. 
Assume by contradiction that $m_{\delta}<C$ as $\delta \searrow \delta_{0}$. We then obtain a sequence $\{ (\delta_{n}, \lambda_{n}, w_{n})\}$ such that $\delta_n\searrow \delta_{0}$, $\lambda_n \in [0, \lambda_1(g_{\delta_n}))$, $\mu_{2,n}^{+}:=\mu_{2}^{+}(\lambda_{n},w_{n})<C$, and $w_n$ is a positive solution of \eqref{pr:w} with $g=g_{\delta_{n}}$ for $\lambda_{n}\in  [0,\lambda_1(g_{\delta_{n}}))$. Thus, $\| w_{n}\|_{C(\overline{\Omega})}\leq C_{\delta_{n}}$ using $C_{\delta_{n}}$ given by Lemma \ref{lem:vaw} with $\delta=\delta_{n}$.  We then deduce 
\begin{align*}
\int_{\Omega}|\nabla \phi_{2,n}|^2 
&=\lambda_n \int_{\partial\Omega} g_{n}(x)\phi_{2,n}^2 + 
\mu_{2,n}^{+}\int_{\partial\Omega} g_{n}(x)w_{n}^{p-1}\phi_{2,n}^2 \\ 
&\leq \lambda_n \int_{\partial\Omega} g_{n}(x)\phi_{2,n}^2 + 
C C_{\delta_{n}}^{p-1}\int_{\partial\Omega} g_{n}^{+}(x)\phi_{2,n}^2, 
\end{align*}
where $\phi_{2,n}=\phi_{2}(\mu_{2,n}^{+})$ and $g_n=g_{\delta_n}$. 
By Lemma \ref{lem:vaw} and \eqref{lam1del0}, passing to the limit shows that $\int_{\Omega} |\nabla \phi_{2,n}|^2 \longrightarrow 0$. Since $\| \phi_{2,n} \|=1$, we infer that up to a subsequence and for a constant $c_{\infty}\neq 0$, 
\begin{align*}
&\phi_{2,n}\longrightarrow c_{\infty} \ \mbox{ in }  H^1(\Omega) 
\quad \mbox{(implying that this is the case in $L^{p+1}(\partial\Omega)$)}, \\
& \phi_{2,n} \longrightarrow c_{\infty} \ \mbox{ a.e.\ in $\Omega$ and $\partial\Omega$}. 
\end{align*}
 
We then deduce that 
\begin{align} \label{convC}
\phi_{2,n} \longrightarrow c_{\infty} \quad\mbox{ in } C(\overline{\Omega}). 
\end{align}
Once this is done, we obtain that $\phi_{2,n}$ has constant sign if $n$ is sufficiently large, which is the desired contradiction. Let us show how to deduce \eqref{convC}. 
Since $c_{\infty}$ is a constant, we infer 
\begin{align} \label{pr+1}
\begin{cases}
(-\Delta+1)(\phi_{2,n}-c_{\infty})
= \phi_{2,n}-c_{\infty} & \mbox{ in } 
\Omega, \\
\dfrac{\partial}{\partial \nu}(\phi_{2,n}-c_{\infty})
=\lambda_n g_{n}(x)\phi_{2,n}+\mu_{2,n}^{+}\, 
g_{n}(x)w_{n}^{p-1}\phi_{2,n} & \mbox{ on } \partial\Omega. 
\end{cases}
\end{align}
It should be noted that $\| \phi_{2,n} \|_{C(\overline{\Omega})}$ is bounded, 
based on the bootstrap argument used in the proof of Corollary \ref{cor:bif}. Thus, Lebesgue's dominated convergence theorem applies, and for a fixed $r>N$ we deduce 
\[ 
\phi_{2,n} - c_{\infty} \longrightarrow 0  \quad\mbox{ in } L^{r}(\Omega). 
\]
Similarly, using Lemma \ref{lem:vaw} and the assumption that $\lambda_n \to 0$ and $\mu_{2,n}^+$ is bounded, we deduce that 
\begin{align*} 
\lambda_n g_{n}(x)\phi_{2,n} + \mu_{2,n}^{+}\, g_{n}(x)w_{n}^{p-1}\phi_{2,n} 
\longrightarrow 0 \quad\mbox{ in } L^{r}(\partial\Omega). 
\end{align*} 
The $W^{1,r}$-estimate (\cite[(3.3)Proposition]{Am76N}) applies to \eqref{pr+1}, 
and we infer that 
\begin{align*} 
\phi_{2,n} - c_{\infty} \longrightarrow 0 \quad \mbox{ in } W^{1,r}(\Omega),  
\end{align*}
from which \eqref{convC} follows by Sobolev's embedding theorem. 
\end{proof}

\noindent{\it End of proof of Theorem \ref{thm2}.} \ 
From Proposition \ref{prop:2nd}, we find that 
given $p>1$ satisfying \eqref{psubcr}, 
it is possible to choose $\overline{\delta}>\delta_{0}$ 
such that if $\delta \in (\delta_{0},\overline{\delta})$, then 
$p<\mu_{2}^{+}$ for a positive solution $w$ of \eqref{pr:w} with 
$\lambda\in [0,\lambda_1(g_{\delta}))$. 
The proof of Theorem \ref{thm2} is now complete. \qed  

\begin{rem}
A similar strategy as in the proof of Theorem \ref{thm2} was applied in \cite{KRQU2020} to an indefinite sublinear elliptic equation with a Robin boundary condition. 
\end{rem}

We then use \eqref{change} to convert Theorem \ref{thm2} to \eqref{sppr}, 
which is our main result for \eqref{sppr}. 
\begin{theorem} \label{thm4} 
Let $g_{\delta}\in C^{1+\theta}(\partial\Omega)$ be given by \eqref{gdel}. 
If we take $\delta > \delta_0$ sufficiently close to $\delta_{0}$, then the positive solution set of \eqref{sppr} with $g=g_{\delta}$ 
for $\lambda \in (0,\lambda_1(g_{\delta}))$ is described as 
follows (see Figure \ref{figcurve}):  
\begin{enumerate} \setlength{\itemsep}{0.2cm} 

\item \eqref{sppr} with $g=g_{\delta}$ possesses a unique positive solution $v_{\lambda}$ for each $\lambda\in (0,\lambda_1(g_{\delta}))$, and the positive solution set $\{ (\lambda,v_{\lambda}) :  \lambda\in (0,\lambda_1(g_{\delta})) \}$ is represented by a smooth curve.
\item $v_{\lambda}\rightarrow 0$ in $C^2(\overline{\Omega})$ 
as $\lambda\nearrow \lambda_1(g_{\delta})$, i.e., bifurcation from $\{ (\lambda,0)\}$ at $(\lambda_1(g_{\delta}),0)$ occurs subcritically. 
\item $\lambda^{\frac{1}{p-1}}v_{\lambda}\rightarrow w_{0}$ 
in $C^2(\overline{\Omega})$ as $\lambda\searrow 0$ 
for some $w_{0}\in C^2(\overline{\Omega})$, where $w_{0}$ 
is a unique positive solution of \eqref{pr:w} with 
$g=g_{\delta}$ for $\lambda=0$. 
\end{enumerate}
\end{theorem}

From Theorem \ref{thm:f1} (iv), similar results as those in Theorem \ref{thm2} can be deduced for \eqref{pr:w:f}.  Analogously to \eqref{defGam+}, we introduce for $f$  
the condition 
\begin{align} \label{f:pm}
\partial\Omega = \Gamma_{+}(f)\cup \Gamma_{-}(f) \quad \mbox{ and } \quad  \overline{\Gamma_{+}(f)}\cap \overline{\Gamma_{-}(f)}=\emptyset. 
\end{align}

Then, we have the following: 
\begin{theorem} \label{thm2:gf}
Let $f\in C^{1+\theta}(\partial\Omega)$ be a sign changing function that satisfies \eqref{f:pm} and the condition  $\int_{\partial\Omega}f>0$. 
Let $g_{\delta}\in C^{1+\theta}(\partial\Omega)$ be given by \eqref{gdel}. If $\delta > \delta_0$ is sufficiently close to $\delta_{0}$, then the positive solution set of \eqref{pr:w:f} with  $g=g_{\delta}$ for $\lambda \in (0,\lambda_1(g_{\delta}))$ is given as follows (see Figure \ref{figcurvefg}): 
\begin{enumerate} \setlength{\itemsep}{0.2cm} 
\item \eqref{pr:w:f} with $g=g_{\delta}$ possesses a unique positive solution $w_{\lambda}$ for every $\lambda\in (0,\lambda_1(g_{\delta}))$, and the positive solution set $\{ (\lambda,w_{\lambda}) :  \lambda\in (0,\lambda_1(g_{\delta})) \}$ is represented by a smooth curve, 
\item $w_{\lambda}\rightarrow 0$ in $C^2(\overline{\Omega})$ 
as $\lambda\nearrow \lambda_1(g_{\delta})$ and $\lambda\searrow 0$. 
\end{enumerate}
\end{theorem}

\begin{proof} 
Because of \eqref{vphi1gotoconst}, we find that the condition $\int_{\partial\Omega}f>0$ is sufficient for getting $\int_{\partial\Omega} f(x) \varphi_1(g_{\delta})^{p+1}>0$ for $\delta > \delta_0$ close to $\delta_0$, thus ensuring the existence of the positive solution $W_{\lambda}$ of \eqref{pr:w:f} with $g=g_{\delta}$ for $\lambda\in (0, \lambda_1(g_{\delta}))$ by Theorem \ref{thm:f1}. 

Assertion (ii) is a direct consequence of assertions (ii) and (iv) of Theorem \ref{thm:f1}. 

To verify assertion (i), we will show that the assertion similar to \eqref{xnmax} holds. 
We can deduce that $x_n\in \Gamma_{+}(f)$ for $n$ large enough where $w_n(x_n)=\max_{\partial\Omega}w_n \rightarrow \infty$ as in \eqref{xnmax}, which is essential for our procedure. 
It should be noted that $\frac{\partial w_n}{\partial \nu}(x_n)>0$ by the strong maximum principle and boundary point lemma. By virtue of \eqref{f:pm}, if we assume $x_n \in \Gamma_{-}(f)$, then $f(x_n)<-c_0$ for some $c_0>0$, thus, for a sufficiently large $n$,  \begin{align*}
\frac{\partial w_n}{\partial \nu}(x_n) = \left( \lambda_{n} \, g_{\delta_{n}}(x_n) + f(x_n)\, (w_n(x_n))^{p-1} \right) w_n(x_n) 
< 0, 
\end{align*} 
which is a contradiction. The rest of the proof of assertion (i) follows the same line of argument in this section. 
\end{proof}
%
    \begin{figure}[!htb]
    \centering 
    \includegraphics[scale=0.16]{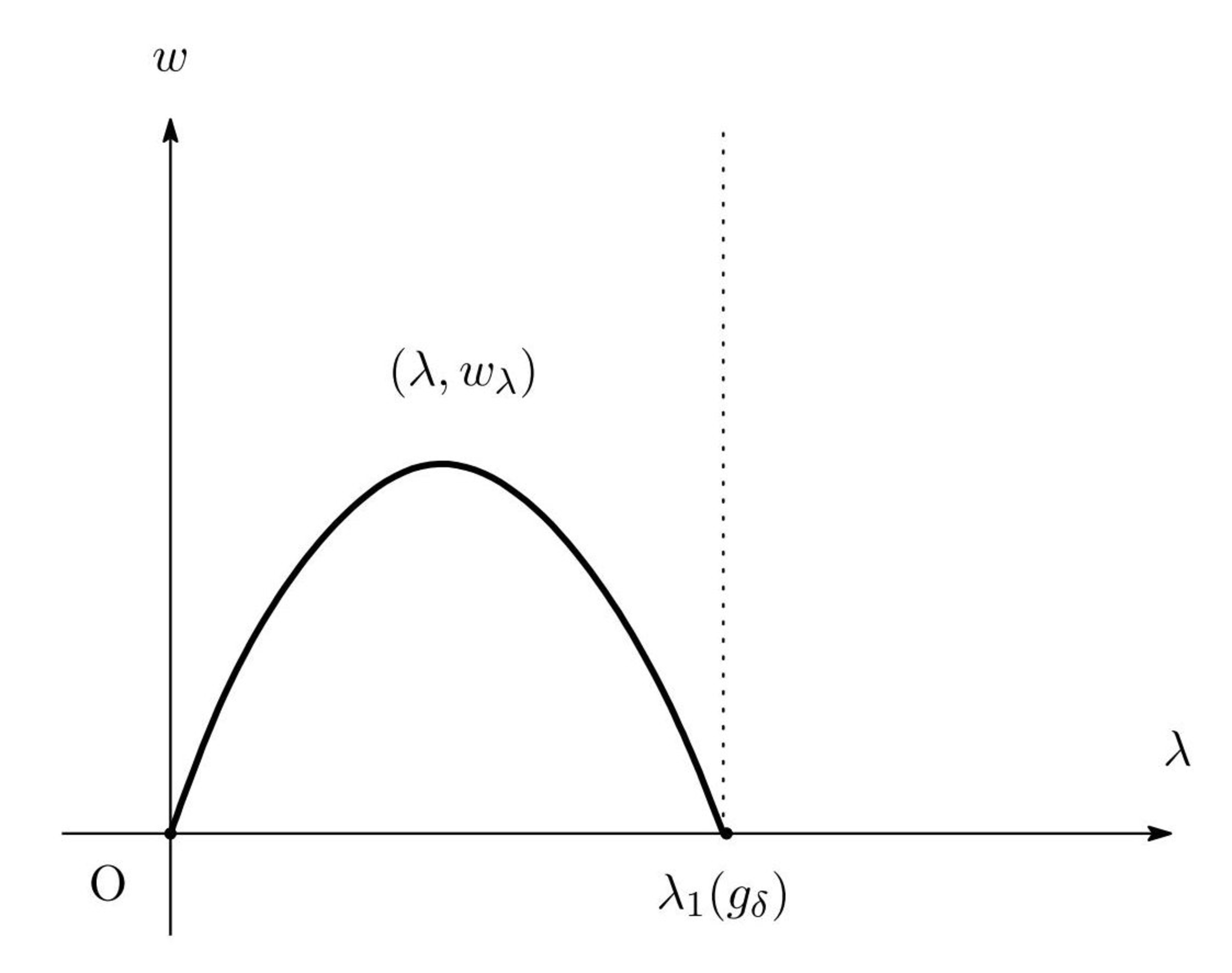} 
		  \caption{Positive solution set $\{ (\lambda,w_\lambda)\}$ of \eqref{pr:w:f} with $g=g_{\delta}$ as a smooth curve in the case that $\delta$ is close to $\delta_{0}$.}
		\label{figcurvefg} 
    \end{figure} 
%

\section{Applications to indefinite  logistic boundary conditions}

\label{sec:Appl}

Let $\Omega$ be a bounded domain of $\mathbb{R}^N$, $N\geq1$, 
with smooth boundary $\partial\Omega$. Consider nonnegative solutions of the problem 
\begin{align} \label{lpr}
\begin{cases}
-\Delta u = 0 & \mbox{ in } \Omega, \\
\dfrac{\partial u}{\partial \nu}  
= \lambda r(x)u(1-u) & \mbox{ on } \partial \Omega,  
\end{cases}
\end{align}
where $r\in C^{1+\theta}(\partial\Omega)$, $\theta\in (0,1)$, changes sign, 
$\lambda > 0$ is a parameter, and $\nu$ is the unit outer normal to $\partial\Omega$.  A motivation for our study of \eqref{lpr} arises in population genetics (\cite{Fl1975, BH90}). For previous works on the boundary version, we refer to \cite{M2009na, MN2011jde, KLS2015, MN2016jmaa}. 
Clearly, $u\equiv 0,1$ satisfies \eqref{lpr} for all $\lambda>0$, which are called \textit{constant solutions}, and  
\begin{align*}
& \{ (\lambda, 0)\}:= \{ (\lambda,u): \lambda>0, \; u\equiv 0\}, \\
& \{ (\lambda, 1)\}:= \{ (\lambda,u): \lambda>0, \; u\equiv 1\}
\end{align*}
are said 
to be the \textit{constant lines}. A positive solution  
$u \in C^{2+\alpha}(\overline{\Omega})$, $\alpha\in (0,1)$,  
of \eqref{lpr} is defined similarly. We may regard $(\lambda,u)$ as a positive solution of \eqref{lpr}. It should be noted that $(\lambda, 1)$ is a (constant) positive solution. 

In this section, we discuss the existence and uniqueness of \textit{nonconstant} positive solutions $u$ of \eqref{lpr}, which have been well studied for the case that 
$u\leq1$ in $\overline{\Omega}$ (\cite{M2009na, MN2011jde, KLS2015, MN2016jmaa}). In turn, our objective is to discuss the case of $u\not\leq 1$ in $\overline{\Omega}$, i.e., the case that $u>1$ somewhere in $\overline{\Omega}$. Such positive solutions are called \textit{large positive solutions}. 
Similar studies on large positive solutions are \cite{KB2000, DS2004, AB2006, RQS2015}, where the logistic type equation $-\Delta u = \lambda r(x)(u-u^p)$ in $\Omega$, $p>1$, is considered under linear Dirichlet, Neumann or Robin boundary conditions. For $\lambda<0$, $(\lambda, u)$ is a positive solution of \eqref{lpr} if and only if $(-\lambda,u)$ is a positive solution of \eqref{lpr} with $r$ replaced by $-r$, because of the symmetry $\lambda r(x) = (-\lambda)(-r(x))$. For $\lambda=0$, it is clear that $\{ (0,c) \}:= \{ (\lambda,u): \lambda=0, \; u\equiv c>0 \,  \mbox{ is a constant} \}$ is the positive solution set.

\subsection{Known results for positive solutions $u\leq1$}

\label{subsec:kruleq1}

In this subsection, we focus our consideration on nonnegative solutions $u$ of \eqref{lpr} such that $u\leq1$ in $\overline{\Omega}$, and summarize the known results from \cite{MN2011jde, MN2016jmaa}, as illustrated by Figure \ref{figuleq1m}. By applying the strong maximum principle and boundary point lemma, a nonconstant positive solution $u$ implies that $0<u<1$ in $\overline{\Omega}$. 

When $\int_{\partial\Omega}r\neq 0$, in view of positive solutions 
$u(x)\in (0,1)$ in $\overline{\Omega}$, it suffices to consider the case of  $\int_{\partial\Omega}r<0$ because all the results obtained for this case
can be converted automatically into the case of $\int_{\partial\Omega}r>0$. 
Indeed, using the change $U=1-u$, we transform \eqref{lpr} into the problem 
\begin{align*}
\begin{cases}
-\Delta U = 0 & \mbox{ in } \Omega, \\
\dfrac{\partial U}{\partial \nu} = \lambda (-r(x))U(1-U) & \mbox{ on } \partial\Omega.
\end{cases}
\end{align*}
Let us now assume that $\int_{\partial\Omega}r<0$. In terms of 
bifurcation from the constant line $\{ (\lambda,0)\}$, an important role is played by the positive principal eigenvalue $\lambda_1(r)>0$ of \eqref{epro01} with $g=r$. 
The local bifurcation theory from simple eigenvalues by  Crandall and Rabinowitz \cite{CR71} shows that a nonconstant positive solution $u$ of \eqref{lpr} bifurcates from 
$\{ (\lambda, 0)\}$ uniquely at $(\lambda_1(r),0)$, which is {\it supercritical}, i.e., in the direction $\lambda>\lambda_1(r)$ (see Figure \ref{figuleq1m} (i)). Then, by employing the implicit function theorem, $u=u_\lambda$ is parametrized smoothly by $\lambda\in (\lambda_1(r), \infty)$ (\cite[Theorems 2.5 and 2.8]{MN2011jde}). Since 
the bifurcating positive solution curve at $(\lambda_1(r),0)$ is unique, 
we see that $(\lambda, u_\lambda)$, $\lambda\in (\lambda_1(r),\infty)$, 
is a unique nonconstant positive solution of \eqref{lpr}  (\cite[Theorem 2.7]{MN2011jde}). 
For $\lambda\leq\lambda_1(r)$, \eqref{lpr} has no nonconstant positive solution (\cite[Theorem 2.6]{MN2011jde}). For the stability of the nonnegative solutions, $(\lambda,0)$ is asymptotically stable and unstable for $\lambda\leq\lambda_1(r)$ and $\lambda>\lambda_1(r)$, respectively, $(\lambda,1)$ is unstable for $\lambda>0$, and $(\lambda,u_{\lambda})$ is asymptotically stable for $\lambda>\lambda_1(r)$ (\cite[Theorem 3.1, 3.2 and 3.3]{MN2011jde}).

When $\int_{\partial\Omega}r=0$, it is proved by a bifurcation approach from $\{ (0,c)\}$ that 
\eqref{lpr} has a unique nonconstant positive solution $u_{\lambda}$ for each $\lambda>0$, which bifurcates from $\{ (0,c)\}$ at $c=\frac{1}{2}$ and is parametrized smoothly by $\lambda\in (0,\infty)$ (\cite[Theorem 1.1, Lemma 2.2]{MN2016jmaa}). For the stability of the nonnegative solutions, $(\lambda,0)$ and $(\lambda,1)$ are both unstable for $\lambda>0$, whereas $(\lambda,u_{\lambda})$ is asymptotically stable for $\lambda>0$ (\cite[Theorem 1.2]{MN2016jmaa}). 
%
		 \begin{figure}[!htb]
        \centering 	  	   
	\includegraphics[scale=0.15]{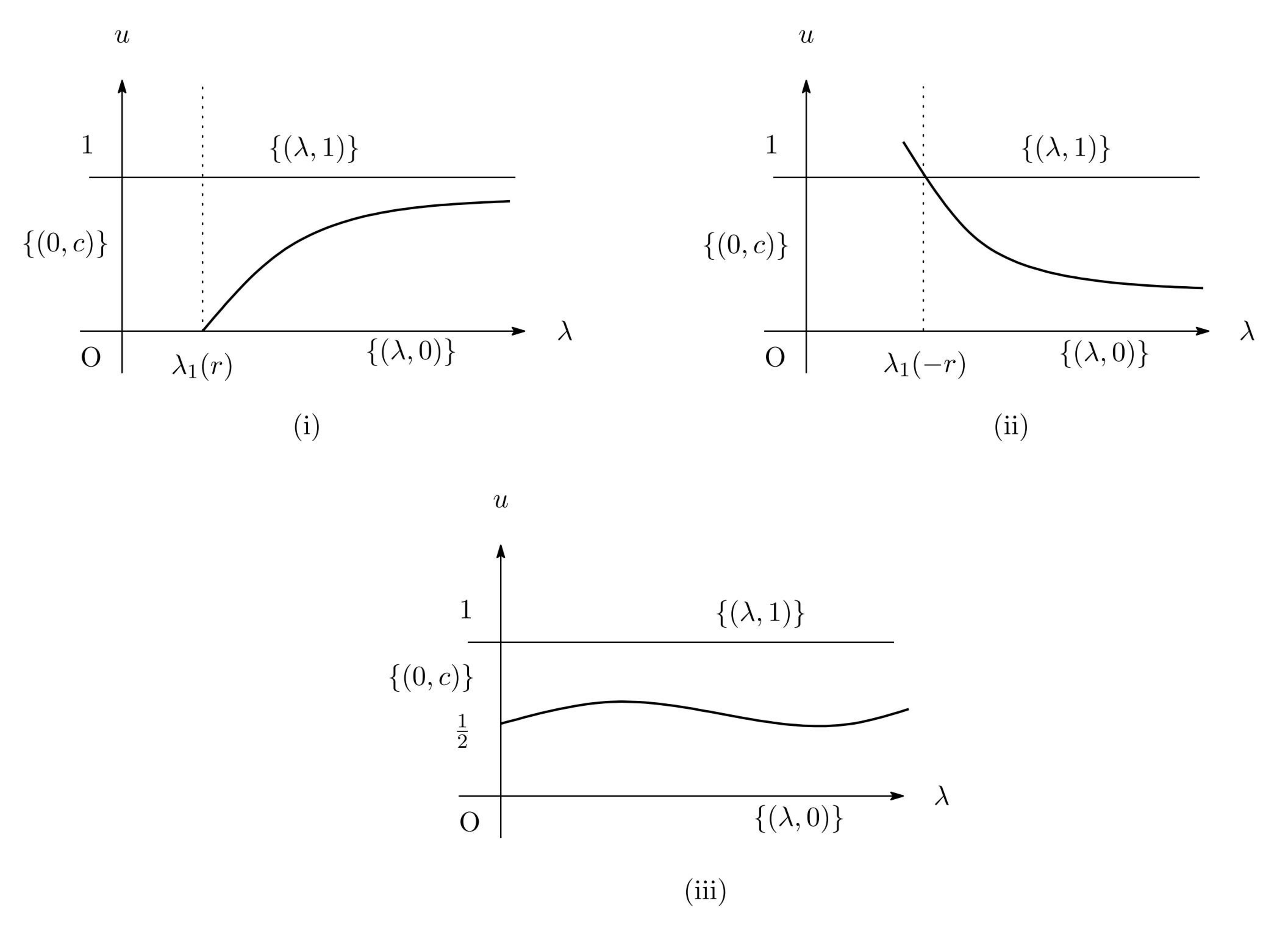} 
		  \caption{Cases (i) $\int_{\partial\Omega}r<0$, (ii) $\int_{\partial\Omega}r>0$, and (iii)  $\int_{\partial\Omega}r=0$.} 
		\label{figuleq1m} 
		    \end{figure}

To conclude this subsection, let us demonstrate our 
developing scenario presented in the next subsection for large positive solutions of \eqref{lpr}. When $\int_{\partial\Omega}r>0$, using the change $v=u-1$, 
another bifurcation approach from the constant line $\{ (\lambda,1)\}$ can be employed. We then transform \eqref{lpr} into \eqref{sppr} with $g=-r$ and $p=2$:
\begin{align} \label{lpr:v}
\begin{cases}
-\Delta v = 0 & \mbox{ in } \Omega, \\
\dfrac{\partial v}{\partial \nu} 
= \lambda (-r(x))(v+v^2) & \mbox{ on } \partial \Omega.  
\end{cases}
\end{align}
The local bifurcation theory also applies to \eqref{lpr:v} at $(\lambda_1(-r),0)$, and \eqref{lpr:v} possesses a smooth solution curve $\{ (\lambda(\sigma),v(\sigma)): \sigma\in (-\overline{\sigma},\overline{\sigma}) \}$,  $\overline{\sigma}>0$, 
which uniquely bifurcates at $(\lambda_1(-r),0)$, such that 
$(\lambda(0),v(0))=(\lambda_1(-r),0)$, 
$\lambda(\sigma)>\lambda_1(-r)$ and $v(\sigma)<0$ in $\overline{\Omega}$ for $\sigma\in (-\overline{\sigma},0)$, and  $\lambda(\sigma)<\lambda_1(-r)$ and $v(\sigma)>0$ in 
$\overline{\Omega}$ for $\sigma\in (0,\overline{\sigma})$. It should be noted that the curve $\{ (\lambda(\sigma), v(\sigma)): \sigma \in  (-\overline{\sigma},0)\}$ corresponds to the positive solutions of \eqref{lpr} with the condition that $u<1$ in $\overline{\Omega}$, whereas the curve $\{ (\lambda(\sigma), v(\sigma)): \sigma \in  (0,\overline{\sigma})\}$ corresponds to large positive solutions of \eqref{lpr} that we want to discuss. The latter result will be strengthened by employing Theorems \ref{thm3} and \ref{thm4}.

\subsection{Main results for large positive solutions $u\not\leq1$} 

In this subsection, we consider large positive solutions of \eqref{lpr}.  The large positive solution $u$ is divided into the two cases: 
\begin{enumerate} \setlength{\itemsep}{0.2cm} 
\item[(a)] $u\geq1$ in $\overline{\Omega}$,  and  
\item[(b)] $u\not\geq1$  in $\overline{\Omega}$ ($u<1$ somewhere in $\overline{\Omega}$). 
\end{enumerate}
Here, case (a) implies that $u>1$ in $\overline{\Omega}$, i.e., 
$(\lambda,v)$ with $v=u-1$ is a positive solution of \eqref{lpr:v}, 
whereas case (b) implies that $(\lambda, v)$ is a sign changing solution of \eqref{lpr:v}.

When $\int_{\partial\Omega}r>0$, 
the positive solution curve $\{ (\lambda(\sigma), v(\sigma)): \sigma \in  (0,\overline{\sigma})\}$ of \eqref{lpr:v} added to $(\lambda_1(-r),0)$ 
is extended globally in $\mathbb{R}\times C(\overline{\Omega})$ as a subcontinuum $\mathcal{C}_0$, using the global bifurcation result (\cite[Theorem 6.4.3]{LGbook}, \cite[Theorem 1.1]{Um2010}).  
Indeed, $\mathcal{C}_0\setminus \{ (\lambda_1(-r),0)\}$ does not meet any point on $\{ (\lambda, 0)\}$ by the uniqueness of the bifurcation point $(\lambda_1(-r),0)$. Moreover, equipped with $g=-r$, Lemma \ref{prop:non} and Proposition \ref{prop:bbelow} show that 
\begin{align*}
\{ \lambda\geq0 : (\lambda,v)\in \mathcal{C}_0 \setminus \{ (\lambda_1(-r),0)\} \} \subset (0,\lambda_1(-r)).     
\end{align*}
Particularly, $\mathcal{C}_0$ is unbounded and bifurcates from infinity at some $\lambda\in [0,\lambda_1(-r)]$. We then deduce that \eqref{lpr} possesses the unbounded subcontinuum 
\begin{align*}
\tilde{\mathcal{C}}_0=\{ (\lambda, u) : u=1+v, (\lambda,v)\in \mathcal{C}_0 \}
\end{align*}
where $\tilde{\mathcal{C}}_0$ bifurcates at $(\lambda_1(-r),1)$, and 
$\tilde{\mathcal{C}}_0\setminus \{ (\lambda_1(-r),1)\}$ consists of large positive solutions with condition (a). Our aim is to develop this global bifurcation result by employing Theorems \ref{thm3} and \ref{thm4}.

We then present our main results for large positive solutions of \eqref{lpr} in the case of $\int_{\partial\Omega}r>0$, where \eqref{psubcr} is assumed with $p=2$, i.e., $N=1,2,3$.

\begin{theorem} \label{thm:A1}
Let $N=1,2,3$. Suppose that $\int_{\partial\Omega}r>0$. Then, \eqref{lpr} possesses at least one large positive solution $u_{\lambda}\in C^2(\overline{\Omega})$ with condition (a) for each $\lambda\in (0,\lambda_1(-r))$. 
Moreover, there exist no large positive solutions of \eqref{lpr} with condition (a) for any $\lambda\geq \lambda_1(-r)$ nor with condition (b) for any $\lambda\in (0,\lambda_1(-r)]$. 
Additionally, the following {\rm five} assertions hold: 
\begin{enumerate} \setlength{\itemsep}{0.2cm} 
\item $u_{\lambda}$ is unstable. 
\item $\| u_{\lambda} \|_{C(\overline{\Omega})} \rightarrow \infty$  as $\lambda\searrow 0$. More precisely,  
there exists $0<c_1<c_2$ such that 
\begin{align} \label{asymptC}
c_1\lambda^{-1} \leq \| u_{\lambda} \|_{C(\overline{\Omega})}\leq c_2 \lambda^{-1} \ \ \mbox{ as } \ \ \lambda\searrow 0.  
\end{align}
\item $\sup_{\lambda\in (\underline{\lambda}, \, \lambda_1(-r))}\| u_{\lambda} \|_{C(\overline{\Omega})}<\infty$ for any $\underline{\lambda}\in (0, \lambda_1(-r))$. 
\item $\|u_{\lambda}-1\|_{C^2(\overline{\Omega})} \rightarrow 0$ as $\lambda\nearrow \lambda_1(-r)$. Particularly, $(\lambda, u_{\lambda})$ is the bifurcating positive solution on $\tilde{\mathcal{C}}_0$ for $\lambda < \lambda_1(-r)$ that is close to $\lambda_1(-r)$. 
\item $\inf_{\lambda\in (0, \, \overline{\lambda})}\| u_\lambda -1 \|_{C(\overline{\Omega})}>0$ for any $\overline{\lambda}\in (0,\lambda_1(-r))$. 
\end{enumerate}
\end{theorem}

The next theorem is a direct consequence of Theorems \ref{thm4} and \ref{thm:A1}. 

\begin{theorem} \label{thm:A2}
Let $N=1,2,3$. 
Let $r\in C^{1+\theta}(\partial\Omega)$ be such that \eqref{G} holds with $g=-r$, and let $r_{\delta}=\delta r^{+}-r^{-}$,  $\delta>\delta_0:=\frac{\int_{\partial\Omega}r^-}{\int_{\partial\Omega}r^+}$, be a function in $C^{1+\theta}(\partial\Omega)$. 
Then, the large positive solution $(\lambda ,u_\lambda)$ of \eqref{lpr} with $r=r_{\delta}$, given by Theorem \ref{thm:A1}, satisfies the following {\rm three} properties, provided that 
$\delta > \delta_0$ is sufficiently close to $\delta_{0}$ (see Figure \ref{figexact}):  

\begin{enumerate} \setlength{\itemsep}{0.2cm} 
\item $u_{\lambda}$ is a unique large positive solution (and consequently a unique nonconstant positive solution), 
and the positive solution set $\{ (\lambda,u_{\lambda}) : \lambda\in (0,\lambda_1(-r_{\delta})) \}$ is represented by a smooth curve. 
\item $u_{\lambda}\rightarrow 1$ in $C^2(\overline{\Omega})$ 
as $\lambda\nearrow \lambda_1(-r_{\delta})$. 
\item $\lambda u_{\lambda} \rightarrow w_{0}$ 
in $C^2(\overline{\Omega})$ as $\lambda\searrow 0$ 
for some $w_{0}\in C^2(\overline{\Omega})$, where $w_{0}$ 
is a unique positive solution of \eqref{pr:w} with 
$g=-r_{\delta}$ and $p=2$ for $\lambda=0$. 

\end{enumerate}
\end{theorem}

\begin{rem} 
Without any smallness condition on $\delta-\delta_0$, \eqref{lpr} with $r=r_{\delta}$ has a {\it unique} nonconstant positive solution $u_\lambda$ for $\lambda \neq \lambda_1(-r_\delta)$ but near $\lambda_1(-r_{\delta})$, satisfying $0<u_{\lambda}<1$ in $\overline{\Omega}$ for $\lambda>\lambda_1(-r_{\delta})$ and $u_{\lambda}>1$ in $\overline{\Omega}$ for $\lambda<\lambda_1(-r_{\delta})$. Moreover, there is no nonconstant positive solution of \eqref{lpr} for $\lambda=\lambda_1(-r_{\delta})$ (cf.\ Remark \ref {rem:uniq}). 
\end{rem}
%
		 \begin{figure}[!htb]
        \centering 	  	   	
        \includegraphics[scale=0.15]{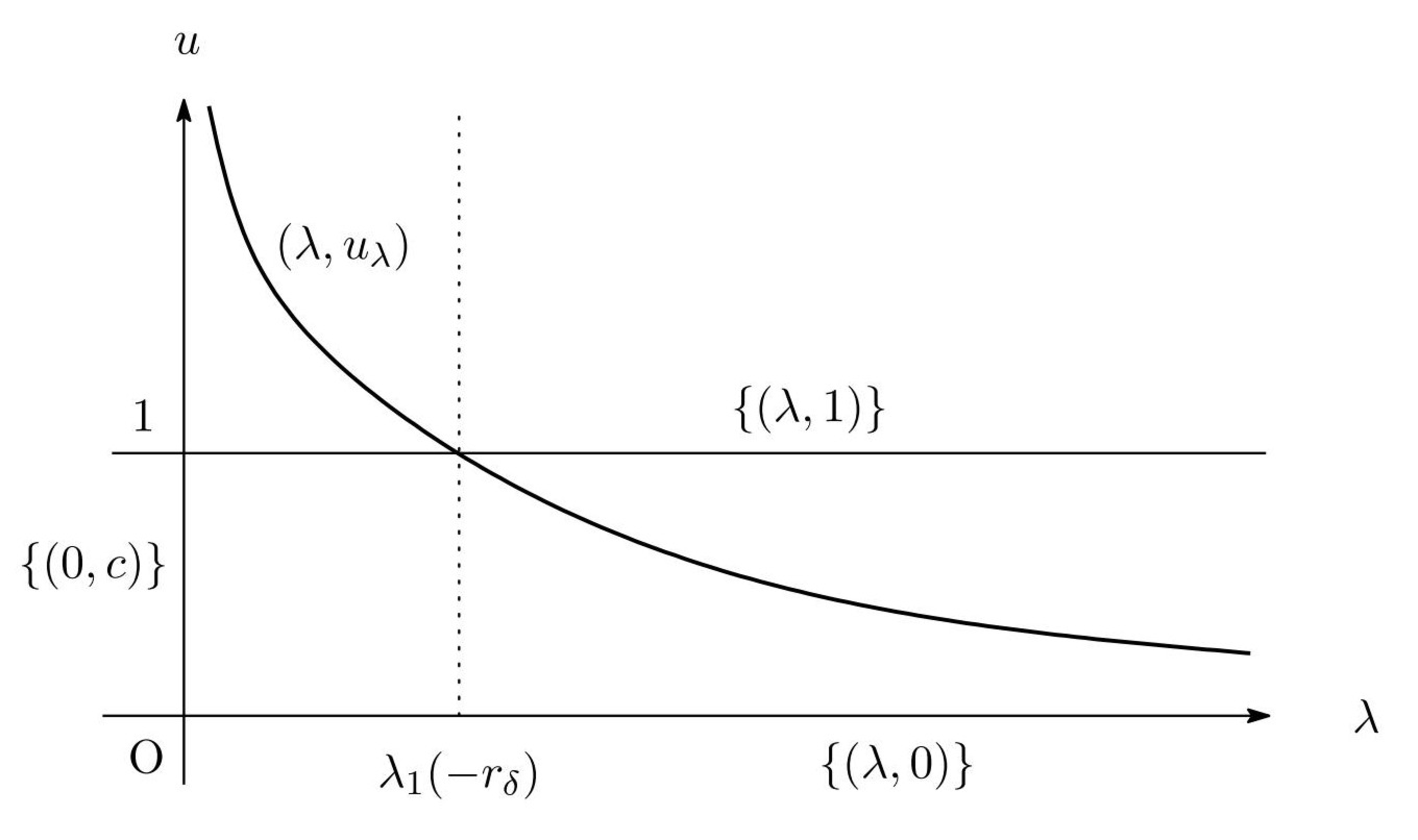} 
		  \caption{Positive solution curve of \eqref{lpr} for $r=r_{\delta}$ with $\delta$ larger than but close to $\delta_0$. }
		\label{figexact} 
		    \end{figure} 

When $\int_{\partial\Omega}r<0$, motivated by \eqref{asymptC}, 
the following asymptotics condition for positive solutions of \eqref{lpr} 
is introduced when $\lambda\searrow 0$: there exists $c_0>0$ such that 
\begin{align} \label{asymptA}
\| u_n \|_{C(\overline{\Omega})}\leq c_0 \lambda_n^{-1}, \quad n\to \infty,  
\end{align}
for a sequence $\{ (\lambda_n, u_n) \}$ of positive solutions of \eqref{lpr} that satisfies $\| u_n  \|_{C(\overline{\Omega})}\rightarrow \infty$ as $\lambda_n\searrow 0$. 

We then state our main result for the case of $\int_{\partial\Omega}r\leq0$. 

\begin{theorem} \label{thm:A3}
Let $N\geq1$ be arbitrary. Assume that $\int_{\partial\Omega}r\leq0$. 
Then, \eqref{lpr} has no large positive solution with condition (a) for any $\lambda>0$. If $\int_{\partial\Omega}r<0$, then additionally assuming $N=1,2,3$ and \eqref{asymptA} provides that $(\lambda, 1)$ is the only positive solution of \eqref{lpr} for small $\lambda>0$. 
\end{theorem} 

\begin{rem} 
To obtain \eqref{asymptA}, it is sufficient that \eqref{G} holds with $g=-r$, which follows from the same argument as given in the proof of Lemma \ref{lem:vaw}. Indeed, we can deduce the existence of $c_0>0$ such that $\| \lambda u \|_{C(\overline{\Omega})}\leq c_0$ for a positive solution $u$ of \eqref{lpr} for $\lambda \searrow 0$. 
\end{rem}

\subsection{Proof of Theorem \ref{thm:A1}}

The assertions in Theorem \ref{thm:A1} are deduced directly from Theorem \ref{thm3}, except for the nonexistence assertion for the large positive solutions with condition (b). We verify it using the sub and supersolution method (\cite[(2.1) Theorem]{Am76N}). 

Assume by contradiction that $u$ is a large positive solution of \eqref{lpr} with condition (b) for some $\lambda\in (0,\lambda_1(-r)]$. Note that $u-1$ changes sign. 
Say $\overline{u}:=\min (u,1)\in C(\overline{\Omega})$, and we reduce \eqref{lpr} to a fixed point equation for the continuous and compact mapping associated with \eqref{lpr} in the ordered Banach space $C(\overline{\Omega})$, where the mapping is strongly increasing in the order interval $[0, \overline{u}]$. This reduction is in the same spirit as in \cite[Section 2]{Um2010}. Then, we can show that $\overline{u}$ is a supersolution of the fixed point equation 
such that $0<\underline{u}\leq 1$ and $\underline{u}\not\equiv 1$ in $\overline{\Omega}$ (cf.\ \cite[Theorem 1.2]{Li82}). Note that $u\equiv 0$ is an unstable constant solution of \eqref{lpr} for every $\lambda>0$, thus, we can construct a positive subsolution $\underline{u}$ of the fixed point equation such that $0< \underline{u}\leq \overline{u}$ in $\overline{\Omega}$. By the sub and supersolution method, we infer the existence of a nonconstant positive solution $u_1$ of \eqref{lpr} such that $0<\underline{u}\leq u_1 \leq \overline{u}\leq 1$ in $\overline{\Omega}$, which is contradictory for \cite[Theorem 2.6]{MN2011jde}. 
The proof of Theorem \ref{thm:A1} is complete. \qed


\subsection{Proof of Theorem \ref{thm:A3}} 

The first assertion comes from Proposition \ref{prop:non} with $g=-r$. 
We next prove the second assertion. It suffices to verify that \eqref{lpr} has no large positive solution with condition (b) for $\lambda>0$ small. Assume by contradiction that \eqref{lpr} has positive solutions $\{ (\lambda_n, u_n)\}$ with $\lambda_n\searrow 0$ such that $u_n-1$ changes sign, and we deduce that $\| u_n \|_{C(\overline{\Omega})}\rightarrow \infty$. Indeed, if $\| u_n \|_{C(\overline{\Omega})}$ is bounded, then 
$\| u_n \|$ is also bounded, and so is $\| u_n  \|_{C^{\alpha}(\overline{\Omega})}$ by elliptic regularity. By the compactness, we deduce that up to a subsequence, 
$u_n \rightarrow u_\infty$ in $C(\overline{\Omega})$. We then infer that 
$u_\infty\equiv 1$ because $u_n -1$ changes sign and $u_{\infty}$ is a constant. However, this is contradictory, because we can show that $\{ (\lambda,1)\}$ is a unique bifurcation curve from $\{ (0,c)\}$ at $(\lambda,u)=(0,1)$ by a bifurcation approach relying on the Lyapunov and Schmidt reduction of \eqref{lpr}, similarly as in the proof of \cite[Proposition 3.11]{KRQU2020a}.

From \eqref{asymptA}, we obtain that $\| \lambda_n u_n \|_{C(\overline{\Omega})} \leq c_0$. Say $v_n=\lambda_n u_n$, and 
$(\lambda_n, v_n)$ is a positive solution of the problem:
\begin{align*} 
\begin{cases}
-\Delta v = 0 & \mbox{ in } \Omega, \\
\dfrac{\partial v}{\partial \nu}=\lambda r(x)v - r(x)v^2 
& \mbox{ on } \partial\Omega. 
\end{cases}    
\end{align*}
Since $\| v_n\|$ is bounded, for some $v_\infty \in H^1(\Omega)$ and up to a subsequence, $v_n \rightharpoonup v_\infty\geq0$, and $v_n \rightarrow v_\infty$ in $L^2(\Omega)$ and $L^3(\partial\Omega)$. The argument then proceeds by dividing it into the two cases:

(i) Case $v_\infty\not\equiv 0$ on $\partial\Omega$: 
By definition, 
\begin{align} \label{vnwsol}
\int_{\Omega}\nabla v_n \nabla \varphi - \lambda_n \int_{\partial\Omega} r(x)v_n \varphi +  \int_{\partial\Omega}r(x)v_n^2 \varphi=0,  
\quad \varphi \in H^1(\Omega). 
\end{align}
Passing to the limit, we deduce 
\begin{align*} 
\int_{\Omega}\nabla v_\infty \nabla \varphi +   \int_{\partial\Omega}r(x)v_\infty^2 \varphi=0.  
\end{align*}
Thus, by elliptic regularity we deduce that $v_\infty \in C^2(\overline{\Omega})$, and so $v_\infty>0$ in $\overline{\Omega}$ by the strong maximum principle and boundary point lemma, which is a contradiction because this implies $\int_{\partial\Omega}r>0$. 

(ii) Case $v_\infty\equiv 0$ on $\partial\Omega$: 
We then deduce that $v_n \rightarrow 0$ in $H^1(\Omega)$ because \eqref{vnwsol} with $\varphi=v_n$ implies 
\begin{align*}
\int_{\Omega}|\nabla v_n|^2 = \lambda_n \int_{\partial\Omega} r(x)v_n^2 
-\int_{\partial\Omega} r(x)v_n^3 \longrightarrow 0. 
\end{align*}
Say $\psi_n = \frac{v_n}{\| v_n\|}$, and $\| \psi_n \|=1$. Up to a subsequence and for some $\psi_{\infty}\in H^1(\Omega)$, $\psi_n \rightharpoonup \psi_\infty$, and $\psi_n \rightarrow \psi_\infty$ in $L^2(\Omega)$ and $L^3(\partial\Omega)$. Since 
\begin{align*}
\int_{\Omega} |\nabla \psi_n|^2=\lambda_n \int_{\partial\Omega}r \psi_n^2 - \| v_n\| \int_{\partial\Omega}r \psi_n^3\longrightarrow 0, 
\end{align*}
we infer that $\psi_n \rightarrow \psi_\infty$ in $H^1(\Omega)$, and $\psi_\infty$ is a positive constant. From \eqref{vnwsol} with $\varphi=1$, we infer that 
\begin{align*}
\int_{\partial\Omega} \lambda_n r(x) \psi_n = \int_{\partial\Omega} r(x) \psi_n^2 \| v_n\|. 
\end{align*}
Using $v_n = \lambda_n u_n$, we obtain that 
\begin{align*}
\int_{\partial\Omega}r(x)\psi_n = \| u_n \| \int_{\partial\Omega}r \psi_n^2.     
\end{align*}
It should be noted that $\| u_n \|\rightarrow \infty$ from the condition $\| u_n\|_{C(\overline{\Omega})}\rightarrow \infty$ by elliptic 
regularity. Passing to the limit, we deduce 
\begin{align*}
&\int_{\partial\Omega}r \psi_n \longrightarrow \psi_\infty  \int_{\partial\Omega}r < 0, \\ 
&\| u_n \| \int_{\partial\Omega}r \psi_n^2\longrightarrow -\infty, 
\end{align*}
which is a contradiction. The proof of Theorem \ref{thm:A3} is now complete. \qed

\subsection{Large positive solutions in the one dimensional case}

\label{subsec:1D}

In this subsection, we consider the one dimensional case of \eqref{lpr}: 
\begin{align} \label{lpr1}
\begin{cases}
-u^{\prime\prime}(x)=0, \quad  x\in I=(0,1), & \\
-u^{\prime}(0)=\lambda r_0 \, u(0)(1-u(0)), & \\
u^{\prime}(1)=\lambda r_1 \, u(1)(1-u(1)), & 
\end{cases}
\end{align}
where $\lambda>0$ is a parameter, and $r_0<0<r_1$. 
Then, it is necessary that the positive solutions $u$ are linear functions $u=cx+d$ with the constants $c$ and $d$ that satisfy $d>0$ and $c+d>0$. So, problem \eqref{lpr1} is reduced to solving the equation
\begin{align} \label{1prob}
\begin{cases}
&-c = \lambda r_0 d (1-d), \\
&c = \lambda r_1 (c+d)(1-c-d). 
\end{cases}
\end{align}

\begin{prop}\label{prop:1D}
Problem \eqref{lpr1} does not have any large positive solution with condition (b). 
\end{prop}

\begin{proof}
We argue by contradiction. Let $u=cx+d$ be a positive solution of \eqref{lpr1} for some $\lambda>0$
 such that 
\begin{align*}
u(x_{-}):=\min_{x=0,1}u(x)<1<\max_{x=0,1}u(x)=:u(x_{+}). 
\end{align*} 
Then, it follows that either $(x_{-}, x_{+})=(0,1)$ or $(1,0)$. In both cases, the second assertion in \eqref{1prob} does not hold.
\end{proof}


Consequently, any nonconstant positive solution $(\lambda, u)$ of \eqref{lpr1} fulfills that either $u>1$ in $[0,1]$ or $0<u<1$ in $[0,1]$. Combining this assertion with Theorems \ref{thm:A1} and \ref{thm:A2} then provides the following corollary. It is understood that $r(0)=r_0$ and $r(1)=r_1$. 

\begin{cor} \label{cor:1dim}
The following {\rm two} assertions hold:  
\begin{enumerate} \setlength{\itemsep}{0.2cm} 
\item If $r_0+r_1\leq0$, then \eqref{lpr1} has a unique nonconstant positive solution $u_{\lambda}$ for $\lambda>\lambda_1(r)$, satisfying $0<u_{\lambda}<1$ in $[0,1]$, and there is no nonconstant positive solution of \eqref{lpr1} for any $\lambda\in (0,\lambda_1(r)]$. 
\item If $r_0+r_1>0$, then \eqref{lpr1} has a nonconstant positive solution $u_{\lambda}$ for $\lambda\neq \lambda_1(-r)$ and no nonconstant positive solution for $\lambda=\lambda_1(-r)$ such that $0<u_{\lambda}<1$ in $[0,1]$ for $\lambda>\lambda_1(-r)$, whereas $u_{\lambda}>1$ in $[0,1]$ for $\lambda\in (0,\lambda_1(-r))$. 
Additionally, $u_{\lambda}$ is unique for $\lambda>\lambda_1(-r)$. Moreover, given $r_0<0$, if we choose $r_1>0$ such that $r_0+r_1>0$ is sufficiently close to zero, then $u_{\lambda}$ is also unique for $\lambda\in (0,\lambda_1(-r))$. 
\end{enumerate}
\end{cor}

A numerical observation of \eqref{lpr1} assists the understanding of the nonconstant positive solution set. 
This observation suggests that a nonconstant positive solution $u_{\lambda}$ is always unique if it exists, and one obtains the following limiting behaviors of $u_{\lambda}$ both as $\lambda\to\infty$ and $\lambda\searrow 0$:

(i) in each case of $r_0, r_1$, the limiting function of $u_{\lambda}$ as $\lambda\to\infty$ is given by $u_{\infty}(x)=x$, which is consistent with \cite[Theorem 4.1]{MN2011jde} and \cite[Theorem 1.3]{MN2016jmaa}; 

(ii) in the case when $r_0+r_1>0$, the limiting 
function of $\lambda u_{\lambda}$ as $\lambda\searrow 0$ is given by 
\begin{align*}
w_0(x) = -\frac{\left( 1-\sqrt{-\frac{r_0}{r_1}}\right)^2}{(-r_0)}x +\frac{1-\sqrt{-\frac{r_0}{r_1}}}{(-r_0)}, 
\end{align*}
which is the unique positive solution of the problem 
\begin{align*}
\begin{cases}
-w''=0 \quad \mbox{ in } (0,1), & \\
-w'(0)=-r_0 \, w(0)^2, & \\ 
w'(1)=-r_1 \, w(1)^2,  
\end{cases}    
\end{align*}
and thus, this is consistent with Theorem \ref{thm:A2} (iii).

To conclude this paper, we mention that when $N>1$, it is an open question whether or not \eqref{lpr} has bifurcating positive solutions from $\{ (\lambda,1)\}$ at a positive {\it non principal} eigenvalues of \eqref{epro01} with $g=-r$. 
If yes, then Proposition \ref{prop:non} tells us that these bifurcating positive solutions should be large positive solutions with condition (b). Otherwise, Corollary \ref{cor:1dim} could be extended to $N>1$ in a certain class of sign changing weights $r$. In the case of $r<0$ on $\partial\Omega$ (sign definite superlinear case), the existence of such bifurcating positive solutions 
was discussed in \cite{KLS2015}.


\end{document}